\documentclass[journal,twoside,web]{ieeecolor}
\pdfoutput=1
\usepackage{generic}
\usepackage{cite}
\usepackage{amsmath,amssymb,amsfonts}
\usepackage{algorithmic}
\usepackage{graphicx}
\usepackage{algorithm,algorithmic}
\usepackage{mathtools,mathbbol,dsfont}
\usepackage{bm}
\usepackage{textcomp}
\usepackage[all]{xy}
\usepackage{chemarrow}
\usepackage{color}
\usepackage{mathrsfs,array}
\makeatletter

\newcommand{\Rmnum}[1]{\expandafter\@slowromancap\romannumeral #1@}
\makeatother
\newtheorem{theorem}{Theorem}

\newtheorem{lemma}{Lemma}
\newtheorem{definition}{Definition}
\newtheorem{proposition}[theorem]{Proposition}
\newtheorem{remark}{Remark}
\newtheorem{example}{Example}
\usepackage{textcomp}
\def\BibTeX{{\rm B\kern-.05em{\sc i\kern-.025em b}\kern-.08em
    T\kern-.1667em\lower.7ex\hbox{E}\kern-.125emX}}
\begin{document}
\title{Automatic Implementation of Neural Networks through Reaction Networks---Part \Rmnum{1}: Circuit Design and Convergence Analysis}
\author{Yuzhen Fan, Xiaoyu Zhang, Chuanhou Gao, \IEEEmembership{Senior Member, IEEE}, and Denis Dochain
\thanks{This work was funded by the National Nature Science Foundation of China under Grant No. 12320101001, 12071428, and 62303409, the China Postdoctoral Science Foundation under Grant No. 2023M733115. This work is an extension of our earlier paper in the 4th IFAC Workshop on Thermodynamics Foundations of
Mathematical Systems Theory, Jul. 25-27, 2022, Canada.} 
\thanks{Y. Fan and C. Gao are with the School of Mathematical Sciences, X. Zhang is with the College of Control Science and Engineering, Zhejiang University, Hangzhou, China (e-mail: yuzhen$\_$f@zju.edu.cn, gaochou@zju.edu.cn (correspondence), Xiaoyu\_Z@zju.edu.cn). }
\thanks{D. Dochain is with 
ICTEAM, UCLouvain, B ˆatiment Euler, avenue Georges Lemaˆıtre 4-6, 1348 Louvain-la-Neuve, Belgium (e-mail: denis.dochain@uclouvain.be).}}

\maketitle
 
\begin{abstract}
Information processing relying on biochemical interactions in the cellular environment is essential for biological organisms. The implementation of molecular computational systems holds significant interest and potential in the fields of synthetic biology and molecular computation. This two-part article aims to introduce a programmable biochemical reaction network (BCRN) system endowed with mass action kinetics that realizes the fully connected neural network (FCNN) and has the potential to act automatically in vivo. In part \Rmnum{1}, the feedforward propagation computation, the backpropagation component, and all bridging processes of FCNN are ingeniously designed as specific BCRN modules based on their dynamics. This approach addresses a design gap in the biochemical assignment module and judgment termination module and provides a novel precise and robust realization of bi-molecular reactions for the learning process. Through equilibrium approaching, we demonstrate that the designed BCRN system achieves FCNN functionality with exponential convergence to target computational results, thereby enhancing the theoretical support for such work. Finally, the performance of this construction is further evaluated on two typical logic classification problems.
\end{abstract}


\section{Introduction}
\label{sec:introduction}
\IEEEPARstart{L}{iving} cells of natural organisms have impressive capabilities of responding to cellular environments by adapting themselves to external stimuli. Indeed, these goals are powered by the sophisticated dynamical behaviors resulting from biomolecular interactions involved in cells, such as gene regulatory networks, signal pathway networks, and so on. Synthetic biology is a booming field of utilizing engineering principles to design biological components (e.g. oscillators\cite{elowitz2000synthetic}, toggle switches\cite{gardner2000construction}) and assemble them into more complex biological circuits to mimic the life behavior or program cells with novel and practical functionalities for use in other fields, such as medical care, energy, environment and so on \cite{qian2018programming}. More recently, the move toward a design based on a ``system level'' in this field made creating sufficiently complicated biochemical reaction circuits to accomplish more tasks in the real world possible. 


One type of common task being carried out focusing on the enhancement of robustness includes regulating protein production under unintended interactions \cite{qian2021robustness}, designing molecular controllers to reject disturbances and noise in cellular environments\cite{aoki2019universal, xiao2018robust}, etc., which are always finished by engineered molecular control modules and meanwhile reveal the regulatory characteristics to maintain homeostasis in living organisms. Another one refers to the information-processing function that is also essential to biological organisms, for example, recognizing molecular patterns \cite{cherry2018scaling} and decision-making in vivo by acquiring and processing multiple biological signal inputs. Related nature examples can be referenced to bacteria following chemical gradients (chemotaxis) \cite{wadhams2004making} and the brain distinguishing complex odor information \cite{mori1999olfactory}. In synthetic biology, such missions have a wide range of applications in medical care, such as disease diagnosis \cite{zhang2020cancer}, and they
could be performed by dint of molecular computation, namely embedding computation into the context compatible with the cellular environment and traditional computers cannot go. It requires us to develop the computational capabilities of biochemical kinetics for designing molecular computation modules and further implementing complicated biochemical information-processing systems.

Chemical reaction network (CRN) is a widely recognized mathematical model to describe biomolecular interactions. Moreover, it has been demonstrated that arbitrary CRN endowed with mass action kinetics could be realized as DNA strand displacement reactions \cite{soloveichik2010dna}, which implies that abstract CRN model could always be given physical entities. Therefore, it is regarded as a widely used bridge to design biological circuits at the level of dynamical systems, especially used as a ``programming language'' \cite{vasic2020crn++} in molecular computation due to the sequence-specific nature and high-density information storage capacity of DNA molecules. Nevertheless, compiling a molecular computing system to perform complex computational or information-processing tasks remains a challenging matter. Implementing basic mathematical operations with CRNs such as addition, subtraction, multiplication, and division has been well-studied \cite{buisman2009computing} but it is still difficult to program CRNs directly on this basis for complicated tasks such as classification and decision-making. Besides, living systems could exhibit adaptive behaviors to changes in the environment while most synthetic circuits are only designed for predefined goals, so constructing biological systems capable of intelligent learning ability is of great interest and reward, especially applied in smart therapeutics \cite{lakin2023design,zhao2023synthetic,yan2023applications,auslander2012smart} to identify and adapt to each patient's pattern.

Artificial neural networks have established a mature computational framework to solve various problems in silicon and the parameters could be trained according to data sets by mathematically well-defined optimization algorithms, which is akin to the adaptation of our biological system to outside environments. Therefore, they are regarded as a promising programming paradigm for molecular computation systems capable of learning and performing brain-like information-processing functionalities. Previous work \cite{okumura2022nonlinear,cherry2018scaling,qian2011neural} demonstrated a concrete implementation of the feedforward structure of neural networks using DNA strand displacement reactions in wet labs. Apart from that, theoretical work on programming neural networks through CRN appears following closely and has recently seen rapid development. The perceptron model with ReLU or Sigmoid activation functions was realized by molecular sequestration reaction systems and phosphorylation-dephosphorylation cycles, respectively \cite{samaniego2021signaling,9304515}. A class of rate-independent reaction networks was discovered to implement neurons with the ReLU nonlinearity solely through stoichiometric changes \cite{vasic2022programming}. A general mathematical framework for implementing feedforward computation with a kind of smooth ReLu function was proposed in \cite{anderson2021reaction}. Some of the designed feedforward structures above are simple, lacking in scalability for more practical applications, and all are only for the given weight parameters. To achieve molecular learning circuits, some qualitative results \cite{blount2017feedforward} and quantitative implementation \cite{arredondo2022supervised,lakin2023design} of the weights update part have recently emerged.

Admittedly, excellent results have been achieved in compiling neural networks using CRN systems. Limitations remain in terms of completeness, efficiency of the realization, and theoretical support. Regarding completeness, the automated operation of the biochemical neural networks is not guaranteed currently since the process for feeding training samples into chemical neural networks before every iteration was carried out manually by controlling round time nowadays. In other words, the individual CRN design for the assignment module was not considered in the literature, which should be the key part of implementing a complete molecular learning circuit with the ability to run in the cellular environment without artificial interventions. The quantitative design for the learning part is very limited and mainly focuses on programming the gradient descent algorithm by approximating the derivative operator\cite{arredondo2022supervised,9304515} by the difference format,  which may introduce the corresponding approximation errors and cause a complicated design. Also, the judgment operation to terminate the training process when the preset error precision is satisfied is absent now. Additionally, mass-action models of chemical systems are polynomial dynamical systems and are known to exhibit myriad dynamic behaviors \cite{zhang2020persistence,zhang2023persistence,angeli2021robust,forger2017biological}. It is highly nontrivial to design appropriate mass-action reaction systems leading to desirable steady states and rapid convergence speed \cite{anderson2021reaction} with given parameters and initial conditions. Thus, performing dynamic analysis is essential for assurance of the correct construction and these properties.

In this paper, we propose in Section \ref{sec:module design} a complete biochemical fully connected neural networks design method that promises to work independently in vivo, where we give a brand novel CRN design for the assignment module utilizing linear transformation, realize the Sigmoid activation function via modifying an autocatalytic reaction, and construct the judgment module that could control the termination of training processes based on a bistable reaction system. Also, we put forward a \textit{Split Factor - Multiplication} method to calculate the multiplication of any number of factors and then construct a precise implementation of the mini-batch gradient descent (MBGD) algorithm, where the reaction systems designed for the negative gradient computation can be demonstrated to occur in parallel regardless of the depth of neural networks. 
Our construction could handle any combination ways of factors, be robust to the initial concentrations, and be finished only by bi-molecular reactions with the minimum number of intermediate species (more rational in practice and easier to be realized as DNA strand displacement reactions).
Furthermore, Section \ref{sec:dynamic analysis} provides a rigorous mathematical analysis of the realization performance at the system level. We ensure that our designed biochemical systems implement a neural network and have certain desirable properties, where the feedforward reaction system can converge to the target output exponentially, and the positive equilibrium points of the learning module in the final iteration are the fixed points of the MBGD algorithm with the exponential convergence.  Finally, the biochemical neural network is used in two typical logic classification problems to confirm its efficiency in Section \ref{sec:case study}.

\noindent \textbf{Notations:} $\mathbb{R}^n, \mathbb{R}^n_{\geq0}, \mathbb{R}^n_{>0}, \mathbb{Z}^n_{\geq0}$ denote $n$-dimensional real space, nonnegative space, positive space and nonnegative integer space, respectively. 
$\mathscr{C} (\cdot;*)$ is the set of continuous differentiable functions from $\cdot$ to $*$. For any real value function $f(\cdot)$ and a matrix $A = (a_{ij}) \in \mathbb{R}^{n \times n}$, $f(A) = (f(a_{ij})) \in \mathbb{R}^{n \times n} $ denotes $f(\cdot)$ acts on any element of $A$. Besides, $a_{\cdot j}$ indicates the $j$th column, and $a_{i\cdot}$ represents the $i$th row in $A$. The norm $\Vert \cdot \Vert$ appearing in this paper is regarded as $\infty$-norm in Euclidean space. $\mathds{1}$ and $\mathds{O}$ represents a dimension-suited column vector with all elements to be $1$ and $0$, respectively.

\section{Preliminaries}
\label{sec:pre}
In this section, we formally introduce the basic concepts of CRNs
\cite{feinberg1972complex,feinberg1979lectures,horn1972general}, the fully connected neural network (FCNN) with the specific structure of \textit{two-two-one}, and the computation potential of CRNs as a programming language. 

\subsection{Chemical Reaction Network}
Consider a CRN with $n$ species that interact with each other through $r$ reactions, then the CRN definition takes 
\begin{definition}[CRN]
A CRN consists of the following three finite sets
\begin{enumerate}
    \item \textit{species set:}
    $\mathcal{S}=\{X_{1}, \ldots,X_{n}\}$ that denotes the subjects participated in reactions; 
    \item \textit{complex set:}
    $\mathcal{C}=\bigcup_{j=1}^{r}\{v_{.j},v'_{.j}\}$ with $v_{.j},v'_{.j}\in \mathbb{Z}_{\geq 0}^n$, where each complex is the linear combination of species $\bigcup_{i=1}^{n}X_i$, meaning the $i$th entry of $v_{.j}$, i.e., $v_{ij}$, is the stoichiometric coefficient of $X_i$;
    \item \textit{reaction set:} $\mathcal{R}=\bigcup_{j=1}^{r}\{v_{.j}\to v'_{.j}\}$ satisfying that $\forall v_{.j}\to v'_{.j}\in \mathcal{R}$, $v_{.j}\neq v'_{.j}$, and $\forall v_{.j}\in \mathcal{C}$, $\exists v'_{.j}\in\mathcal{C}$ supporting either $v_{.j}\to v'_{.j}$ or $v'_{.j}\to v_{.j}$.  
\end{enumerate}
We often use the triple $(\mathcal{S},\mathcal{C},\mathcal{R})$ to represent a CRN. 
\end{definition}


Based on the above definition, the $j$th reaction is written as
\begin{equation}
	\sum_{i=1}^{n} v_{ij} X_{i} \longrightarrow \sum_{i=1}^{n} v'_{ij} X_{i},
	\label{eq:1}
\end{equation} 
where $v_{.j}=(v_{1j},...,v_{nj})^\top, v'_{.j}=(v'_{1j},...,v'_{nj})^\top$ are called \textit{reactant complex} and \textit{product complex}, respectively. Further, we define the \textit{reaction vectors} of $j$th reaction by $v'_{.j}-v_{.j}$ and the linear subspace spanned by all reaction vectors, called the \textit{stoichiometric subspace}, by
\begin{equation}\label{stoisub}
\mathscr{S} \triangleq span\{v'_{.j}-v_{.j},j=1,...,r\}.
\end{equation}
The dynamics of $(\mathcal{S},\mathcal{C},\mathcal{R})$ that captures the change of the concentration of each species, labeled by $x\in\mathbb{R}^n_{\geq 0}$, may be reached if a $r$-dimensional vector-valued function $\mathscr{K}\in \mathscr{C}(\mathbb{R}^n,\mathbb{R}^r)$ is defined to evaluate the reaction rates and the balance
law is further utilized, written as
\begin{equation}
    \frac{dx(t)}{dt} = \Gamma \mathscr{K}(x),~~x \in \mathbb{R}^n_{\geq0}.
    \label{eq:2}
\end{equation}
where $\Gamma\in\mathbb{Z}^{n\times r}$ is the \textit{stoichiometric matrix} with the $j$th column $\Gamma_{\cdot j}=v'_{.j}-v_{.j}$. Usually, \textit{mass action kinetics} (MAK) is used to evaluate the reaction rate, which induces the rate of the $j$th reaction as
\begin{equation}\label{mak}
\mathscr{K}_j(x)=  k_j x^{v.j} \triangleq k_j\prod_{i=1}^{n} x_{i}^{v_{ij}},
\end{equation}
where $k_j>0$ represents the rate constant. Then the integral form of the solution to \eqref{eq:2} plus \eqref{mak} is
\begin{equation}\label{dynIntegral}
\begin{split}
    x(t) 
    &=  x(0)+ \sum_{j=1}^r k_j (\int_{0}^{t} \prod_{i=1}^{n} x_{i}^{v_{ij}}(\tau) d\tau) (v'_{.j}-v_{.j}),
\end{split}
\end{equation}
where $x(0) \in \mathbb{R}^n_{\geq0}$ represents any given initial point. Clearly, the CRN system endowed with MAK offers polynomial ordinary differential equations (PODEs), often termed as \textit{mass action system} (MAS) and denoted by the quad $(\mathcal{S},\mathcal{C},\mathcal{R},k)$. The dynamics of (\ref{dynIntegral}) means that any trajectory of a MAS will evolve within a special equivalence class, called \textit{Stoichiometric Compatibility Class}.


\begin{definition}[Stoichiometric Compatibility Class] Given a CRN $(\mathcal{S},\mathcal{C},\mathcal{R})$ and $x_0 \in \mathbb{R}^n_{\geq0}$, the set induced by $\mathcal{P}(x_0)=\{x\in\mathbb{R}^n_{\geq 0}~|~x-x_0\in\mathscr{S}\}$ is called the \textit{stoichiometric compatibility class} of $x_0$, and $\mathcal{P}^+(x_0)=\mathcal{P}(x_0)\cap \mathbb{R}^n_{>0}$ is named the \textit{positive stoichiometric compatibility class} of $x_0$.   
\end{definition}

We further give the definitions of equilibrium and exponential convergence. 
\begin{definition}[Equilibrium]
For a MAS $(\mathcal{S},\mathcal{C},\mathcal{R},k)$ governed by (\ref{eq:2}) plus (\ref{mak}), if a constant vector $\bar{x} \in 
\mathbb{R}^n_{\geq 0}$ or $\forall x_{\mathcal{A}}\in\mathcal{A} \subset \mathbb{R}^n_{\geq 0}$ satisfy $\Gamma \mathscr{K}(\bar{x})=0$ or $\Gamma \mathscr{K}(x_{\mathcal{A}}) =0 $, then $\bar{x}$ is called the \textit{nonnegative equilibrium point} and $\mathcal{A}$ is the \textit{equilibrium set}. Both are collectively known as \textit{equilibrium}.
\end{definition}

\begin{definition}[Exponential Convergence]
Consider a MAS $(\mathcal{S},\mathcal{C},\mathcal{R},k)$ described by (\ref{eq:2}) plus (\ref{mak}) and admitting an equilibrium point $\bar{x}$ or an equilibrium set $\mathcal{A}$. The solution $x(t)$ of this MAS \textit{converges exponentially} to $\bar{x}$ or to $\mathcal{A}$ if there are positive constants $M,\gamma$ \textgreater $0$ supporting $\Vert x(t) - \bar{x} \Vert \leq M e^{-\gamma t}$ or $\inf_{y \in \mathcal{A}} \Vert x(t)-y \Vert \leq M e^{-\gamma t}$ for all $t\geq 0$.
\end{definition}


\subsection{Fully Connected Neural Network}


Assume a FCNN with an input layer, a hidden layer and an output layer, and let it perform a classification task for the data set $\mathbb{D}=\{(x^i,d^i)\}_{i=1}^p$ with $p$ to represent the size of $\mathbb{D}$. The training process utilizes the common gradient descent-based error backpropagation algorithm (GDBP) but with the mini-batch strategy (i.e., at each iteration only part of samples are fed to update weights), and the activation function takes the \textit{Sigmoid function} given by 
\begin{equation}\label{eq:sigmoid}
f(z)=\frac{1}{1+e^{-z}}.
\end{equation}
Mathematically, the training process can be expressed according to feedforward propagation and backward propagation. For simplicity but without the loss of generality, we fix the structure of FCNN to be \textit{two-two-one} nodes in every related layer to exhibit this process.
\subsubsection{Feedforward Propagation}
Denote the sample matrix by $\chi\in\mathbb{R}^{3\times p}$ with the $l$th column to be $\chi_{.i}=(x^i_1,x^i_2,d^i)^\top$, the \textit{input matrix} by $\Xi\in\mathbb{R}^{3 \times \tilde{p}}_{\geq 0}$ with $\Xi_{3.}=\mathds{1}^\top$ to handle dumb nodes corresponding to bias in the input layer and $\tilde{p}$ ($\tilde{p}\leq p$) to represent the mini-batch size for training, and the \textit{weight matrices} connecting input-hidden layers and hidden-output layers by $\mathcal{W}_1\in\mathbb{R}^{2\times 3}$ and $\mathcal{W}_2\in\mathbb{R}^{1\times 3}$ with their last columns to represent the bias attributed to the nodes of hidden and output layers, respectively. Then the feedforward computation process at the $m$th iteration follows
\begin{gather}\label{eq:weightsum}
\begin{split}
\mathcal{N}^m&=\mathcal{W}_1^m\cdot \Xi^m,\Upsilon^m(\mathcal{W}^m_1) =f(\mathcal{N}^m),\tilde{\Upsilon}^m=
    \begin{pmatrix}
        \Upsilon^m\\
        \mathds{1}^\top
\end{pmatrix},\\
\tilde{\mathcal{N}}^m&=\mathcal{W}_2^m\cdot \tilde{\Upsilon}^m, y^m(\mathcal{W}_2^m) =f(\tilde{\mathcal{N}}^m),
 \end{split}
\end{gather}
where $\mathcal{N}\in \mathbb{R}^{2\times \tilde{p}}$, $\Upsilon\in\mathbb{R}^{2\times \tilde{p}}$ is the \textit{hidden matrix}, $\tilde{\Upsilon}\in\mathbb{R}^{3\times \tilde{p}}$ is constructed to add a row of $1$ to $\Upsilon$ to handle dumb nodes in the hidden layer, $\tilde{\mathcal{N}}\in \mathbb{R}^{1\times \tilde{p}}$, and $y\in\mathbb{R}^{1\times \tilde{p}}$ is the \textit{output matrix}. We further can compute the training error 
\begin{equation}
    \mathcal{E}^m(\mathcal{W}_1^m, \mathcal{W}_2^m)=\frac{1}{2}(\delta^\top-y^m)(\delta^\top-y^m)^\top,
    \label{eq:loss fun}
\end{equation} 
where $\delta=(d^1,...,d^{\tilde{p}})^\top$. This finishes the feedforward computation of the $m$th iteration with $\tilde{p}$ samples.
\subsubsection{Backward Propagation}
This process serves to update weights according to GDBP, which takes
\begin{gather}
\begin{split}
    \mathcal{W}^{m+1}&=\mathcal{W}^{m} + \Delta\mathcal{W}^m
       \label{eq:GD}
\end{split}
\end{gather}
with $\mathcal{W}=
    \begin{pmatrix}
        \mathcal{W}_1\\
        \mathcal{W}_2
\end{pmatrix}$, $\eta \in \left(0,1\right]$ to be the learning rate and 
\begin{gather}\label{eq:weight_update}
\begin{split}
\Delta \mathcal{W}_{1_{ij}}^m &= -\eta \frac{\partial \mathcal{E}^m}{\partial  \mathcal{W}^m_{1_{ij}}} = -\eta \sum^{\tilde{p}}_{l=1} e^m_l \frac{\partial f}{\partial n^m_{3l}} \mathcal{W}^m_{2_{1i}} \frac{\partial f}{\partial n^m_{il}} \Xi_{jl}^m, \\
\Delta \mathcal{W}_{2_{1j}}^m &=-\eta \frac{\partial \mathcal{E}^m}{\partial  \mathcal{W}^m_{2_{1j}}} = -\eta \sum^{\tilde{p}}_{l=1} e^m_l \frac{\partial f}{\partial n^m_{3l}} \tilde{\Upsilon}^m_{jl},\\
&~~~~i=1,2;~j=1,2,3.
\end{split}
\end{gather}
Here, $e_l^m=(\delta^{\top} - y^m)_{1l},~(n^m_{1l}, 
n^m_{2l})^{\top} = \mathcal{N}^m_{\cdot l},~n^m_{3l} = \tilde{\mathcal{N}}^m_{1l}$. After finishing updating weights, the training for the $m+1$ iteration fires through feeding another group of $\tilde{p}$ samples from $\mathbb{D}$ into the input layer, and performing the feedforward computation repeatedly. The termination condition is usually based on the precision requirement, i.e., $\vert e^l \vert$ is less than the preset threshold for all $l$, otherwise, the training process continues. We present a complete flow chart in Fig. 1 to exhibit how FCNN works.  

\begin{figure}[!t]
\centerline{\includegraphics[width=\columnwidth]{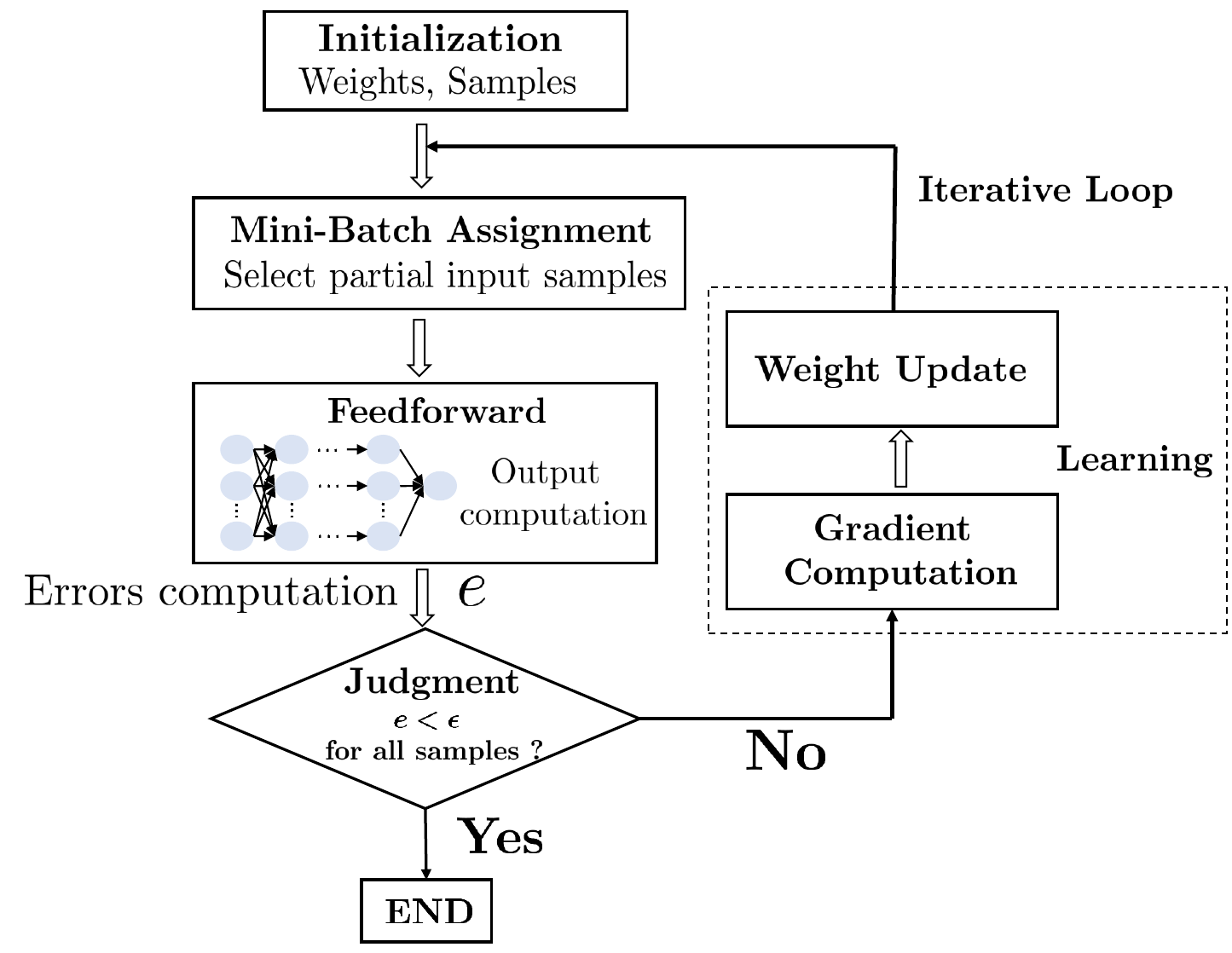}}
\caption{The flow chart of FCNN.}
\label{fig:1}
\end{figure}



\subsection{An Example of Implementing Addition Calculation with CRNs}
There are various calculations, like addition, subtraction, etc, and information processing, like assignment, and logical judgment, etc. in the process of FCNN. Every operation needs to be implemented by CRNs for our current task, and moreover, each variable (a fixed constant after calculation or processing) involved in FCNN needs to be represented by some counterpart in CRN. Thus, for a designed MAS governed by \eqref{eq:2} plus \eqref{mak} with given $x(0) \in \mathbb{R}^n_{\geq0}$,
this prompts the \textit{limiting steady states} \cite{anderson2021reaction} of some species
\begin{equation} \label{def:lss}
  \mathscr{L}[x(t)]= (\lim_{t \to \infty} x_i(t))_{X_i \in \mathcal{S}}
\end{equation}
to represent the corresponding variables in FCNN. Note that \eqref{def:lss} includes globally/locally
asymptoticaly stable (GAS/LAS) equilibrium of partial species if any.
In the following, we give an example of programming addition operation using CRN.      
\begin{example}
\label{exam:addition}
For the addition operation $a+b=c$ in $\mathbb{R}^n_{\geq0}$, we can implement it by the MAS of 
    \begin{gather}
        A \stackrel{1}{\longrightarrow} A+C,~~~ 
       B\stackrel{1}{\longrightarrow}B+C,~~~
       C\stackrel{1}{\longrightarrow}\emptyset,
       \notag
      \end{gather}
where $A$ and $B$ are called \textit{catalysts species} (no concentrations change before and after reactions). The dynamics takes $\dot{x}_c(t)=x_a(t)+x_b(t)-x_c(t)$, which determines the limiting steady state of $C$ to be $\overline{x}_c=\lim_{t \rightarrow \infty}x_c(t)=x_a(0)+x_b(0)$. Therefore, if the initial concentrations of $A$ and $B$ are set by $x_a(0)=a$ and $x_b(0)=b$, and the equilibrium concentration of $C$, i.e., $\bar{x}_c$, represents $c$, then the above network will finish the task of addition $a+b=c$.
\end{example} 


Note that in \textit{Example \ref{exam:addition}} all variables $a,~b,~c$ in addition operation are restricted in $\mathbb{R}^n_{\geq0}$. Thus, they can be expressed by the concentrations (nonnegative attribute) of species naturally. However, in actual calculations there usually involve negative variables. To comply with the nonnegative attribute of species concentration in CRN, 
we adopt the \textit{dual rail encoding} method to make the right match.

\begin{definition}[Dual Rail Encoding \cite{vasic2020deep}]\label{def:dualre}
For any variable $\zeta \in \mathbb{R}$, we take two species $\mathcal Z^{+},\mathcal Z^{-}$ with 
nonegative concentrations $\zeta^{+}(t) \geq0,\zeta^{-}(t) \geq0, \forall t \geq0$, 
between which the difference $\zeta^{+}(t)-\zeta^{-}(t)$ represents the real value $\zeta$.   
\end{definition}

Clearly, the number of variables, i.e., the number of species in CRN, need to be increased doubly after using the technique of dual rail encoding, which will naturally add the complexity of the dynamics of MAS, and further of the implementing process.   

\section{CRNs Programming FCNNs: Circuit Design} \label{sec:module design}
In this section, we use CRNs to program the \textit{two-two-one} structured FCNN utilizing the modular and layered design principle, i.e., dividing FCNN into the modules of assignment, feedforward propagation, judgment and loop, learning and clear-out. The accordingly generated biomolecular circuits are called Biochemical Fully Connected Neural Networks (BFCNNs) in the context. Morevoer, they can be extended to a FCNN with arbitrary depth and width.




Before we perform programming, we introduce the chemical oscillator \cite{vasic2020crn++,shi2022accurate} firstly, which will play an important role in the design of all modules. A major challenge in devising CRNs to perform a series of fussy calculations or information processing is that the latter are executed sequentially while chemical reactions occur in parallel. The chemical oscillator is a feasible solution to address this issue, which, also called molecular clock signal reactions, is a set of chemical reactions admitting a certain oscillatory behavior. The usual oscillatory reactions \cite{lachmann1995computationally} follow
\begin{equation}
    \begin{split}
        O_1+O_2 \stackrel{k_o}\longrightarrow 2 O_2 &,~~
       O_2+O_3 \stackrel{k_o}\longrightarrow 2 O_3, \\       
      \dots\dots ~~~~~~&, ~~O_r+O_1 \stackrel{k_o}\longrightarrow 2 O_1 
    \end{split}
        \label{eq:osci}
\end{equation}
where $O_i,~ i=1,...,r$ denotes clock signal species. We display the concentration curves of a simple chemical oscillator in Fig. \ref{fig:osci}. With certain initial values, each clock signal species has an oscillation phase apart from others in which a species has nonzero concentration while others are zero. We use these signals as catalysts acting on different reaction systems, which will turn on a reaction system but turn off others according to the concentrations (nonzero or zero) of catalysts. This allows to control the occurrence sequence of different reactions without artificial intervention. We will design suitable chemical oscillators to help program almost all modules mentioned above. 
\begin{figure}[!t]
\centerline{\includegraphics[width=\columnwidth]{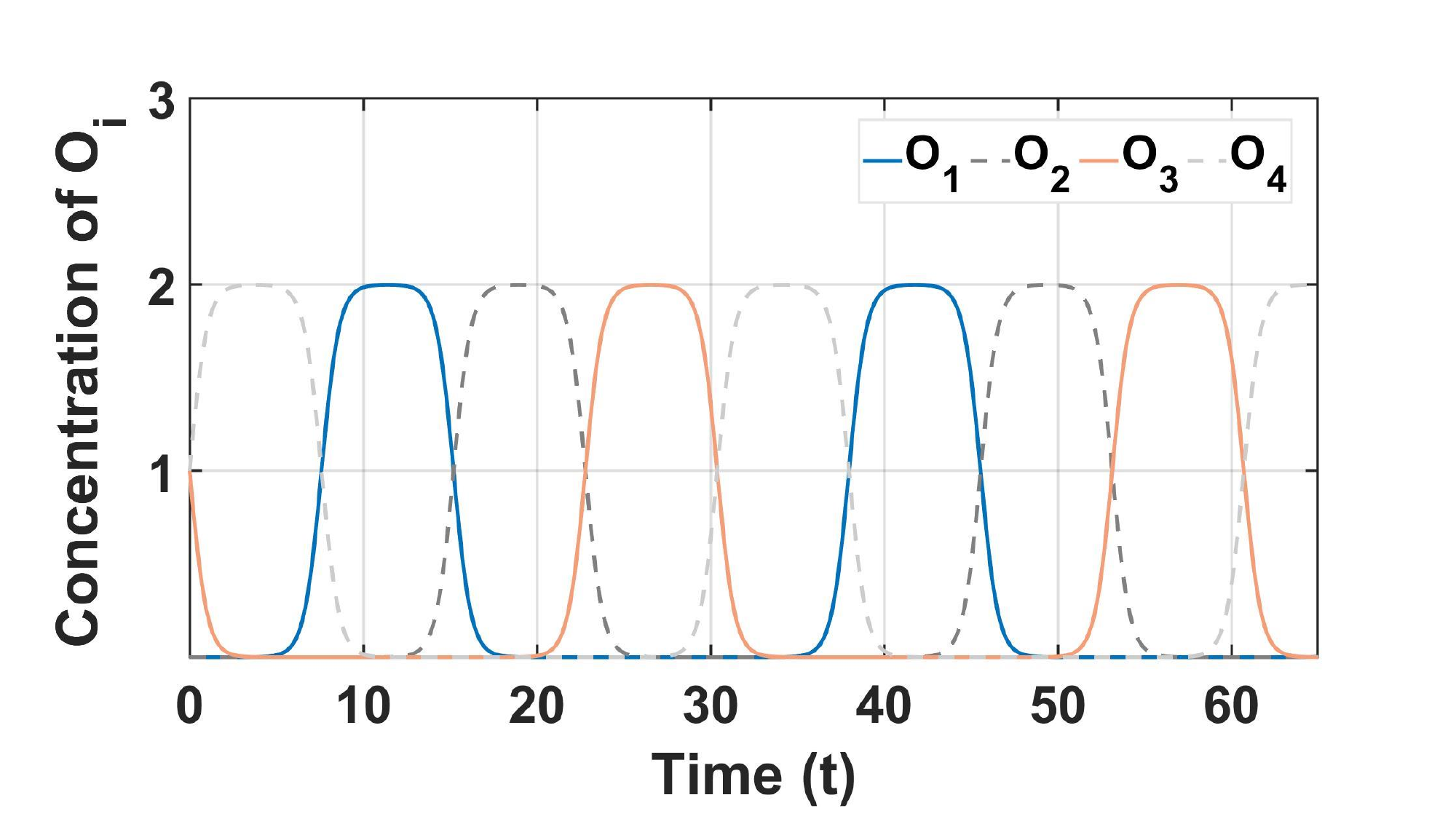}}
\caption{The response curve of the chemical oscillator with four species. Here, $O_1$ and $O_3$ can be used to control two reaction systems.}
\label{fig:osci}
\end{figure}

\subsection{Assignment Module}
From Fig. 1, the first step of FCNN is initialization, i.e., assigning values. We design \textit{Assignment Module} (aBCRN) to complete this task. 

The mini-batch strategy requests to divide the sample matrix $\tilde{\chi}$ into $\frac{p}{\tilde{p}}$ blocks, and feed each block according to $$``block_1 \rightarrow block_2 \rightarrow \cdots block_{\frac{p}{\tilde{p}}} "$$ into the input layer one by one during each iteration. There includes two steps to program this process: one step is to take $\tilde{p}$ samples from original $p$ samples orderly, and the other step is to circulate to feed $``block_1\to...\to block_{\frac{p}{\tilde{p}}}"$ into the input layer again and again. For simplicity but without loss of generality, we assume the samples to be nonnegative, and thus directly introduce the sample species set $\{X^i_1, X^i_2, D^i\}_{i=1}^p$, the input species set $\{S^l_1, S^l_2, S^l_3\}^{\tilde{p}}_{l=1}$, the order species set $\{C^l_i\}$, and the auxiliary order species set $\{\tilde{C}^l_i\}$ to facilitate updating $\{C^l_i\}$ after a round of iteration. The \textit{assignment biochemical reaction network} (aBCRN) operating in three phases is designed as $ \forall i \in \{1,...,p\},~l\in \{1,...,\tilde{p}\}, ~q \in \{1,2\}$ and any constant $k$ 
\begin{gather}
 \mathcal{O}_1:~~~~~  \begin{split}
       C^l_i+X_q^i & \stackrel{1}{\longrightarrow} C^l_i+X_q^i + S^l_q, ~~ S^l_q  \stackrel{1}{\longrightarrow} \varnothing, \\
       C^l_i+D^i & \stackrel{1}{\longrightarrow} 
       C^l_i+D^i +S^l_3, ~~ S^l_3 \stackrel{1}{\longrightarrow} 
       \varnothing,
       \end{split}~~~~
\label{eq:O1}
\\
   \mathcal{O}_3:~~~~~~~~~~~~~~~~~~~~~~  \begin{split}
       C^l_i \stackrel{k}{\longrightarrow} \tilde{C}^l_i,
       \end{split}~~~~~~~~~~~~~~~~~~~~~~~
\label{eq:O3}\\
\mathcal{O}_5:~~ \begin{split}
\tilde{C}^l_o \stackrel{k}{\longrightarrow} C^l_{\tilde{p}+o},
~ \tilde{C}^l_{p-\tilde{p}+l} \stackrel{k}{\longrightarrow}
C^l_l,
~\tilde{C}^l_j  \stackrel{k}{\longrightarrow} C^l_j
       \end{split},~~
\label{eq:O5}
\end{gather}
where $o \in I_l,~ j \notin I_l \bigcup \{p-\tilde{p}+l\}$ with $I_l\triangleq\{l,\tilde{p}+l,...,p-2\tilde{p}+l\}$ for any fixed $l$. 

\begin{remark}
We use $\mathcal{O}_i$ to identify different phases in aBCRN, which means the oscillatory signal $O_i$ produced by (\ref{eq:osci}) is added to all reactions in this phase as a catalyst when aBCRN works. Their addition will not change the dynamics of aBCRN, but can control the occurrence order of different phases. The explanation in this remark will apply to all BCRNs designed in the subsequent.
\end{remark}

In the following, also including the subsequent expressions, we use variables in lower case to represent the concentrations of the corresponding species in upper case to exhibit how aBCRN works. In particular, $c(t)=(c^1(t),\cdots,c^{\tilde{p}}(t))^{\top} \in \mathbb{R}^{p\tilde{p}}_{\geq 0}$, where $c^l(t) = (c^l_1(t),\cdots,c^l_p(t))$ with $c^l_i(t)$ to identify the concentration of $C^l_i$ ($i=1,...,p$).

\begin{proposition}
\label{prop.1}
The aBCRN of \eqref{eq:O1}, \eqref{eq:O3} and \eqref{eq:O5} driven by the corresponding ocsillatory signals can perform automatic assignment through equilibrium if the initial point is set by $c(0) = (e_1,\cdots, e_{\tilde{p}})^\top$ where $e_l \in \mathbb{R}^{1 \times p}$ is a unit row vector with the $l$th element to be 1, $(x_1^i(0),x_2^i(0),d^i(0))^\top=\chi_{.i}$, $\tilde{c}(0) = \mathbb{0}$ and $s(0) \in \mathbb{R}^{3\tilde{p}}_{\geq 0}$ to be any value. 
\end{proposition}

\begin{proof}
The detailed proof appears in Appendix, which may apply to all propositions, lemmas and theorems.
\end{proof}



\subsection{Feedforward Propagation Module}
In this subsection, we design the \textit{Feedforward Propagation Module} (fBCRN), including \textit{linear weighted sum} BCRN (lwsBCRN) and \textit{nonlinear Sigmoid activation BCRN} (sigBCRN). 
\subsubsection{lwsBCRN} This module contributes to programming the computations of $\mathcal{N}^m,~\tilde{\mathcal{N}}^m$ in (\ref{eq:weightsum}). We use the \textit{dual-rail encoding} technique \cite{vasic2020deep} to handle the negative variables, and omit the iteration identification $m$ of all variables if any in Section \ref{sec:pre} for convenience.

For the expression of $\mathcal{N}$ in (\ref{eq:weightsum}), with the technique said in \textit{Definition \ref{def:dualre}} we can rewrite it to be
\begin{gather}\label{eq:n=ws}
\begin{split}
        \mathcal{N}
=
(\mathcal{W}^+_{1}-\mathcal{W}^-_{1}) \cdot \Xi =
        \mathcal{N}^+ 
-
        \mathcal{N}^-  
\end{split},
\end{gather}
where $\mathcal{W}^+_{1}$, $\mathcal{W}^-_{1}$, $\mathcal{N}^+$ and $\mathcal{N}^-$ all contain nonnegative elements. Thus, we introduce the net input species pair $\{N^+_{il},N^-_{il}\}$, the weight species pair $\{W^+_{q}, W^-_{q}\}$ and the bias species pair $\{B^+_{i}, B^-_{i}\}$ with $i=1,2$, $l=1,...,\tilde{p}$ and $q=1,...,4$, and construct the MAS  
\begin{gather}\label{dy:n=w_1x}
\mathcal{O}^1_7:~~~ 
    \begin{aligned}
    W^{\pm}_{i} + S^l_1 & \stackrel{1}{\longrightarrow} W^{\pm}_{i} + S^l_1 + N^{\pm}_{il}, \\ 
     W^{\pm}_{i+2} + S^l_2 & \stackrel{1}{\longrightarrow} W^{\pm}_{i+2} + S^l_2 + N^{\pm}_{il}, \\ 
     B^{\pm}_{i}  & \stackrel{1}{\longrightarrow}  B^{\pm}_{i} +N^{\pm}_{il}, \\
     N^{\pm}_{il}& \stackrel{1}{\longrightarrow} \varnothing. 
\end{aligned}
\end{gather}
Note that the phase identification $\mathcal{O}^1_7$ with the superscript ``1" to represent the first part reaction system controlled by the oscillatory signal $O_7$, and so on. The reaction rate constant is set as $1$ but it can be adjusted as needed actually. The dynamics of \eqref{dy:n=w_1x} thus follows  
\begin{align} \label{eq:21}
    \dot {n}^+ (t) & =  w^+_{1} (t) s_a (t) - n^+(t), \\ \label{eq:22}
     \dot {n}^- (t) & =  w^-_{1} (t) s_a (t) - n^-(t),
\end{align}
where $s_a(t) = (s_{a_{1\cdot}}(t), s_{a_{2\cdot}}(t), \mathds{1}^{\top})^{\top} \in \mathbb{R}^{3 \times \tilde{p}}_{\geq 0}$,
$n^{\pm}(t) \in \mathbb{R}^{2 \times \tilde{p}}_{\geq 0}$ with $n^{\pm}_{il}(t)$ to denote the concentration of $N^{\pm}_{il}$ and $w^{\pm}_{1} (t) \in \mathbb{R}^{2 \times 3}_{\geq 0}$ is the concentration matrix of weight species. Notably, $w^{\pm}_{1_{1\cdot}}(t)$ denotes the concentration of species $W^{\pm}_{1}, W^{\pm}_{3}, B^{\pm}_{1}$ from the left to the right, and $w^{\pm}_{1_{2\cdot}}(t)$ indicates that of species $W^{\pm}_{2}, W^{\pm}_{4}, B^{\pm}_{2}$, respectively. Catalytic reactions (the first three) contribute to increasing the concentration of $N^{\pm}_{il}$ and the decay reaction (the final one) plays the role of balancing. Obviously, \eqref{eq:21}, \eqref{eq:22} are both GAS with the equilibrium points $\bar{n}^+ =  w^+_{1} (0) s_a (0), \bar{n}^- =  w^-_{1} (0) s_a (0)$. Since the oscillator separates reaction modules, $s_{a_1\cdot}(0), s_{a_2\cdot}(0)$ refer to the equilibrium concentrations of species $\{S^l_1, S^l_2\}^{\tilde{p}}_{l=1}$ in phase $\mathcal{O}_1$, which stores the input sample values in the current iteration. Also, let $w^+_{1_{ij}}(0) = \mathcal{W}_{1_{ij}}, w^-_{1_{ij}}(0) = 0$ if $\mathcal{W}_{1_{ij}} >0$, and $w^+_{1_{ij}}(0) = 0, w^-_{1_{ij}}(0) = - \mathcal{W}_{1_{ij}}$ if $\mathcal{W}_{1_{ij}} < 0$. Under these settings, the lwsBCRN system \eqref{dy:n=w_1x} realizes the linear weighted sum computation for the hidden layer through its equilibrium points.
\subsubsection{sigBCRN}
Having computed the net inputs in each hidden layer neuron, we need to pass the values through the nonlinear activation function to achieve a nonlinearity in the output. Compared to current methods based on approximation, we draw inspiration from \cite{arredondo2022supervised} and \cite{fages2017strong} to calculate the accurate Sigmoid output value by devising the sigBCRN system. 

Considering the reversible one-dimensional autocatalytic reaction $\xymatrix{ P \ar @{ -^{>}}^{1} @< 1pt> [r]& 2P \ar  @{ -^{>}}^{1}  @< 1pt> [l] }$
and its dynamic equations $\dot{p} (t) = p(t)(1-p(t))$, the zero deficiency theorem \cite{feinberg1972complex} shows that this MAS has a unique GAS positive equilibrium one and an unstable equilibrium zero, so the solution trajectory of the system always falls within $\mathbb{R}_{>0}$ if given a positive initial concentration. Moreover, the analytic solution shows that $p(t)$ will be less than 1 before reaching equilibrium when $p(0) < 1$, and, in particular, let $p(0)= \frac{1}{2}$, one can have $p(t) = \frac{1}{1+e^{-t}}$.
But this is not the end since, firstly, we need the concentration of net input species to be the independent variable $x$ that will be converted into a nonlinear output signal through $f(x)$. In this regard, we design net input species $N^{\pm}_{il}$ as the catalyst, hydrolyzing autonomously, of the aforementioned reaction. With that, the equilibrium that can be reached by the reactions relies strongly on the initial concentrations of $N^{\pm}_{il}$, which serves to change the independent variable position from $t$ to net input concentrations. Additionally, biochemical reaction systems can only handle the function computation of positive independent variables. Thus, we consider calculating $f(x) = 1/(1+e^{-x})$ when $n^+_{il}-n^-_{il} > 0$, and $\tilde{f}(x) = 1/(1+e^x)$ when $n^+_{il}-n^-_{il} < 0$. It implies that the value and the sign of $n^+_{il}-n^-_{il}$ should be known in advance.

Thus, to construct sigBRN, we introduce hidden-layer output species pairs $\{P^{+}_{il}, P^-_{il}\}$  with $i=1,2$, $l \in \{1,...,\tilde{p}\}$ and species $Half$ maintaining the concentration of $\frac{1}{2}$. The following reactions occupying two phases are for ``clarifying'' the sign of $n_{il}$ in CRN and presetting initial concentrations of $P^{\pm}_{il}$
\begin{gather}
\mathcal{O}_9:~~~~~
\begin{split}\label{CRNinO_9}
    N^+_{il} + N^-_{il} & \stackrel{k}{\longrightarrow} \varnothing,
\end{split}~~~~~~~~~~~~~~~~~~~~\\
\vspace{0.5 em}
\mathcal{O}_{11}:~~
\begin{split}
    N^{\pm}_{il} +  Half & \stackrel{1}{\longrightarrow} N^{\pm}_{il} + Half + P^{\pm}_{il}, \\ \label{eq:O_911}
    N^{\pm}_{il} + P^{\pm}_{il} & \stackrel{1}{\longrightarrow} N^{\pm}_{il}.
\end{split}
\end{gather}
Here, the annihilation reaction \eqref{CRNinO_9} serves to compute the difference, and intuitively speaking, the fixed point $\bar{n}^+_{il}=n_{il}$ if $n_{il} > 0 $ and $\bar{n}^l_{in}=-n_{il}$ in another case after $\mathcal{O}_9$. Owing to the reaction mechanism designed in $\mathcal{O}_{11}$, only the net input species $N^+_{il}$ or $N^-_{il}$ with nonzero initial concentrations can trigger the reactions involving the related species $Half, N^{\pm}_{il}, P^{\pm}_{il}$ in phase $\mathcal{O}_{11}$. Then, the following reaction systems were built occurring in phase $\mathcal{O}_{13}$ to calculate the output $P^+_{il}, P^-_{il}$
\begin{gather}
\mathcal{O}_{13}:
\begin{split}\label{eq:SIG_p}
\xymatrix{ N^+_{il} + P^+_{il} \ar  @{ -^{>}}^{1} @< 1pt> [r] & 2P^+_{il} + N^+_{il}  \ar  @{ -^{>}}^{1}  @< 1pt> [l] },~
N^+_{il} \stackrel{1}{\longrightarrow} \varnothing, ~~~
\end{split} 
\\
\begin{split}\label{eq:SIG_n}
N^-_{il} + 2P^-_{il} & \stackrel{1}{\longrightarrow} 3P^-_{il} + N^-_{il}, \\
N^-_{il} + P^-_{il} & \stackrel{1}{\longrightarrow}  N^-_{il}  \stackrel{1}{\longrightarrow} \varnothing
\end{split}
\end{gather}
for all samples $l$. The system designed above is composed of two uncoupled subnetworks without dynamical interference. The MAS \eqref{eq:SIG_p} is designed to realize $\frac{1}{1+e^{-n^+_{1l}(0)}}$, and \eqref{eq:SIG_n} is a slightly revised version of \eqref{eq:SIG_p} to compute $\frac{1}{1+e^{n^-_{1l}(0)}}$. Recall that, after $\mathcal{O}_{11}$, initial conditions in phase $\mathcal{O}_{13}$ will become $p^+_{il}(0)=\frac{1}{2}, p^-_{il}(0)=0$ if $n_{il}>0$ and $p^+_{il}(0)=0, p^-_{il}(0)=\frac{1}{2}$ if $n_{il}<0$. Under these settings, in phase $\mathcal{O}_{13}$, \eqref{eq:SIG_p} will occur only when the net input $n_{il}$ take positive values by which we could get $f(n_{il})$ in the equilibrium of $P^+_{il}$, and meanwhile $P^-_{il}$ is going to maintain zero concentration in this phase. If $n_{il} < 0$, \eqref{eq:SIG_n} is initiated instead of \eqref{eq:SIG_p} to provide $\tilde{f}(-n_{il})$ in the equilibrium of $P^-_{il}$. In the end, to use one species to denote the final outputs in the hidden layer in either case, the following reactions are added in phase $\mathcal{O}_{13}$ 
\begin{equation}\label{eq:fO_13}
\begin{split}
P^+_{il} \stackrel{1}{\longrightarrow} P^+_{il} + P^l_{i},~~
P^-_{il} \stackrel{1}{\longrightarrow} P^-_{il} + P^l_{i}, ~~
P^l_{i} \stackrel{1}{\longrightarrow} \varnothing,
\end{split}
\end{equation}
where, at least, one of the initial concentrations of $P^+_{il}$ and $P^-_{il}$ must be zero. Therefore, the sigBCRN system \eqref{CRNinO_9}- \eqref{eq:fO_13} implements the Sigmoid function computation for real independent variables. Our design is confirmed quantitatively by the dynamic behavior analysis in the next section. Then, regarding the hidden-output layer, we further introduce new species pairs $\{N^+_{3l}, N^-_{3l}\}$, $\{W^{+}_q, W^{-}_q\}$, and $\{B^+_3, B^-_3\}$ with $q=5,6, l=1,\cdots,\tilde{p}$ and give the following MAS without illustration.
\begin{gather}\label{dy:n=w_2x}
\mathcal{O}^1_{15}:~~~
    \begin{split}
    W^{\pm}_{5} + P^l_{1} & \stackrel{1}{\longrightarrow} W^{\pm}_{5} + P^l_{1} + N^{\pm}_{3l}, \\ 
     W^{\pm}_{6} + P^l_{2} & \stackrel{1}{\longrightarrow} W^{\pm}_{6} + P^l_{2} + N^{\pm}_{3l}, \\ 
     B^{\pm}_{3}  & \stackrel{1}{\longrightarrow}  B^{\pm}_{3} +N^{\pm}_{3l}, \\
     N^{\pm}_{3l}& \stackrel{1}{\longrightarrow} \varnothing. 
\end{split}
\end{gather}

Consequently, applying the sigBCRN design on species $\{N^{\pm}_{3l}, Half, Y^{\pm}_l, Y^l\}^{\tilde{p}}_{l=1}$, the $y_l$ in FCNN is exhibited by the equilibrium concentration $\bar{y}_l$ of the output species $Y^l$.

\subsection{Judgment Module}
One tough problem of using BCRNs to program the learning process based on the gradient descent is the compilation of derivative operators, $ \frac{\partial \mathcal{E}}{\partial  \mathcal{W}_{1_{ij}}}, \frac{\partial \mathcal{E}}{\partial  \mathcal{W}_{2_{1j}}}$ with $i=1,2$, $j=1,2,3$. Our method is to compute the backpropagation formula \eqref{eq:weight_update} derived from the chain rule. Notably, due to $\frac{\partial f}{\partial n_{3l}}=y_l(1-y_l), \frac{\partial f}{\partial n_{il}}=\tilde{\Upsilon}_{il}(1-\tilde{\Upsilon}_{il})$, some of the factors in \eqref{eq:weight_update} like $\mathcal{W}_{1_{ij}}, \mathcal{W}_{2_{1j}}, \Xi_{il}, y_l, \tilde{\Upsilon}_{il}$ have been already computed and stored in the corresponding species, while there are no species concentrations to carry other signals that need to be subtracted in advance, i.e., $e_l, 1-\tilde{\Upsilon}_{il}, 1-y_l$. Therefore, this section first involves the reaction network design for \textit{Preceding Calculation Module} with the attempt to store such signals and prepare for the subsequent learning part. Actually, this kind of pre-computing skill is quite ubiquitous in computers to enhance computation efficiency. 

We introduce species set $\{Y^l_e, Y^l_s, \tilde{Y}^l, I^l_y\}$, $\{P^l_{is}, \tilde{P}^l_i, I^l_{p_i}\}$, $\{S^l_{y}, S^l_{p_i}\}$, $\{E^+_l, E^-_l, E_l\}$ for $l \in \{1,\cdots,\tilde{p}\}$ and device the pre-calculation BCRN (pBCRN) as follows
\begin{gather}
\mathcal{O}^1_{23} ~~~~ 
\left\{~~
\begin{split} \label{eq:bound1}
        Y^l & \stackrel{k}{\longrightarrow} Y^l_e + Y^l_s +\tilde{Y}^l \\
        Y^l_e  + S^l_3 & \stackrel{k}{\longrightarrow} \varnothing \\
         Y^l_s  + I^l_y  & \stackrel{k}{\longrightarrow} \varnothing \\
    P^l_i & \stackrel{k}{\longrightarrow} P^l_{is} +\tilde{P}^l_i \\
    P^l_{is} + I^l_{p_i} &\stackrel{k}{\longrightarrow} \varnothing, \\
\end{split}~~~~
\right.~~~
\\
\mathcal{O}^2_{23} ~~~~
 \left\{~~
\begin{split} \label{eq:bound2}
        S^l_3  & \stackrel{1}{\longrightarrow} E^+_l + S^l_3, ~~ E^+_l \stackrel{1}{\longrightarrow} \varnothing \\
        Y^l_e & \stackrel{1}{\longrightarrow} E^-_l + Y^l_e, ~~ E^-_l \stackrel{1}{\longrightarrow} \varnothing \\
        I^l_y & \stackrel{1}{\longrightarrow} S^l_y + I^l_y, ~~S^l_y \stackrel{1}{\longrightarrow} \varnothing \\
    I^l_{p_i} & \stackrel{1}{\longrightarrow} S^l_{p_i} + I^l_{p_i},~~S^l_{p_i} \stackrel{1}{\longrightarrow} \varnothing,
\end{split}~~~~
\right.
\\
\mathcal{O}^3_{23} ~~~~ 
 \left\{~~
\begin{split} \label{eq:bound3}
E^+_{l} & \stackrel{1}{\longrightarrow} E^+_{l} + E_l \\
E^-_{l}  & \stackrel{1}{\longrightarrow} E^-_{l} + E_l \\
E_l & \stackrel{1}{\longrightarrow} \varnothing.
\end{split}~~~~
\right.~~~~~~~~~~~~~~~
\end{gather}
Here, the concentration of species $I^l_y, I^l_{p_i}$ should both maintain constant one. The first three reactions of \eqref{eq:bound1} are designed to measure the difference between output and label $d^l$ as well as between output and one. Specifically, the first reaction with boundary equilibrium points performs a replication operation to assign the concentration of $Y_l$ to $ Y^l_e, Y^l_s, \tilde{Y}^l$. Here, $Y^l_e, Y^l_s$ are used in the subsequent two annihilation reactions for subtraction, and output concentrations could be stored in $\tilde{Y}^l$ for backpropagation. A similar explanation applies to the last two reactions of \eqref{eq:bound1}. The MAS \eqref{eq:bound2} comprises four group catalytic-hydrolysis reactions intending to assign the catalyst concentration to that of degraded species, and so does as \eqref{eq:bound3}.

The pBCRN system works in the following way. Two possible initial concentration relationships between $y_l(0)$ and $s_{3l}(0)$, i.e., $y_l(0) > s_{3l}(0)$ and $y_l(0) < s_{3l}(0)$, relates to the sign of training errors $e_l=d^l-y_l$ in one iteration. Correspondingly, the positive training error is provided in the equilibrium concentrations of species $S^l_3, E^+_l$, the negative one is stored in $Y^l_e, E^-_l$, and species $E_l$ can represent the training error result in either case. Additionally, \eqref{eq:bound1}-\eqref{eq:bound2} works in the same way for computing $1-y_l, 1-\tilde{\Upsilon}_{il}$ that are only stored in equilibrium concentrations of species $S^l_y, S^l_{pi}$ since the $y_l(0), p^l_i(0)$ is always smaller than one due to the Sigmoid output in sigBCRN.



Translating the judgment to biochemical reaction network language corresponds to determining whether the biochemical learning process carried out by follow-up reactions will happen. This part utilizes a bistable biochemical reaction system to perform the \textit{Judgment Module}. Upholding the low-complexity design principles, we use the minimal 2-dimensional bistable MAS, proposed by Wilhelm\cite{wilhelm2009smallest}, with four reactions to assist in our design
\begin{gather}
\mathcal{O}^1_{25}:~~~~ 
\begin{split}\label{eq:bistable}
    A_l & \stackrel{k_1}{\longrightarrow} 2E_l , ~~ 2E_l \stackrel{k_2}{\longrightarrow} E_l+A_l,\\
    E_l+A_l  &\stackrel{k_3}{\longrightarrow} A_l, ~~
    E_l \stackrel{k_4}{\longrightarrow} \varnothing.
\end{split}
\end{gather}
Note that $E_l$ is the error species also participating in pBCRN in phase $\mathcal{O}_{23}$, which serves as input species in this module. Hopf bifurcation can be excluded as the trace of the Jacobian matrix is always negative, and the MAS \eqref{eq:bistable} is dissipative\cite{wilhelm2009smallest}, whose dynamical behavior depends on its rate constants, i.e., there is no positive equilibrium when $D=k_1k_2-4k_3k_4<0$, a saddle-node bifurcation will occur at $D=0$, and the system exhibits bistability if $D>0$. Given specific parameters satisfying $D>0$, the system has three equilibrium points, $(\bar{e}_{l_1}, \bar{a}_{l_1}) = (0,0)$, $(\bar{e}_{l_2}, \bar{a}_{l_2}) = (\frac{k_1k_2-\sqrt{k_1k_2(k_1k_2-4k_3k_4)}}{2k_2k_3},\frac{k_2}{k_1}(\bar{e}_{l_2})^2),~(\bar{e}_{l_3}, \bar{a}_{l_3})=(\frac{k_1k_2+\sqrt{k_1k_2(k_1k_2-4k_3k_4)}}{2k_2k_3}, \frac{k_2}{k_1}(\bar{e}_{l_3})^2)$, which satisfy $\bar{e}_{l_1} < \bar{e}_{l_2} < \bar{e}_{l_3}$. Lyapunov's first method can be used to prove that the first and the third are locally asymptotic stable, and the second is unstable. The following reactions are added so that species $C_a$ inherit the concentration sum of $e_l$ for all input samples and act as the catalysts of subsequent CRNs.  
\begin{gather}
\mathcal{O}^2_{25}:~~~~
\begin{split}\label{eq:add17}
    E_l \stackrel{k}{\longrightarrow} E_l + C_a,~ C_a \stackrel{k}{\longrightarrow} \varnothing,~\forall l \in \{1,...,\tilde{p}\}.    
\end{split}
\end{gather}

We call the  MAS \eqref{eq:bistable} and $\eqref{eq:add17}$ together judgment BCRN (jBCRN) system. It works in the premise that the $\bar{e}_{l_2}$ is set to the preset training threshold by adjusting rate constants and regarding $e_l(0)$ as the absolute value of the error. In this way, if the initial conditions of species are perturbed upwards at $(\bar{e}_{l_2}, \bar{a}_{l_2})$, the system travels to the third value; if they are perturbed downwards, it settles to $(0,0)$. Consequently, only if $\vert e_l \vert$ of all input samples in the current round meet the requirement, the concentration of $C_a$ will reach the zero equilibrium, and the following update reactions will not happen. Thus, the jBCRN system implements the judgment module under these parameter selections. 

We construct the biochemical \textit{Loop module} through chemical oscillators for our current task that attempts to mimic iterative loops in Fig.\ref{fig:1} while running an algorithm if the judgment condition is not achieved. Due to the periodical changes in concentrations of catalyst species, the chemical oscillator can not only control all modules to satisfy sequential execution but also force modules to carry out multiple times, which enables many iterations. Obviously, it is easier and more practical for fully automatic implementation compared to the cell-like compartment used in our previous work \cite{FAN20227}.

\subsection{Learning Module}
Here are two main steps for updating weights using CRNs, i.e., first programming the backpropagation formula for each weight and then getting the new weight species concentrations for the next iteration. This section provides the design of the \textit{Learning Module} composed of \textit{negative gradient} BCRN (ngBCRN) and \textit{update} BCRN (uBCRN), which plays the main role in equipping the biochemical neural networks with learning abilities.

The derivative computation is divided into multiplication and the sum of multiplication concerning samples according to \eqref{eq:weight_update}. 
Common factors existing in the multiplication part for different weights contribute to our realization. In our case, updating nine weights $\mathcal{W}$ produces four types of multiplication formulas. Different types are composed of different numbers of factors, and deeper weight formulas have fewer terms. Specifically, according to the chain rule, $e_{l}y_l(1-y_l)$ for $\mathcal{W}_{2_{13}}$ (3-item) is a global common factor that appears in every formula, and $\tilde{\Upsilon}_{1l}, \tilde{\Upsilon}_{2l}$ is multiplied for $\mathcal{W}_{2_{11}}, \mathcal{W}_{2_{12}}$ (4-item). When error is propagated to the input-hidden connection, $(1-\tilde{\Upsilon}_{1l}) \mathcal{W}_{2_{11}}, (1-\tilde{\Upsilon}_{2l}) \mathcal{W}_{2_{12}}$ will be multiplied again for $\mathcal{W}_{1_{13}}$ and $\mathcal{W}_{1_{23}}$ (6-item), respectively, and $\Xi_{1l}, \Xi_{2l}$ are finally considered for $\mathcal{W}_{1_{11}},\mathcal{W}_{1_{21}}$ and $\mathcal{W}_{1_{12}}, \mathcal{W}_{1_{22}}$ (7-item), respectively. Furthermore, the \textit{dual-rail encoding} for $e_l, \mathcal{W}_2$ will complicate the process. 
Here, we propose the design method, called \textit{Split Factor - Multiplication} (SFP), for computing the multiplication consisting of any number of factors using a MAS consisting of reactions with no more than two reactants. The main computation strategy in SFP refers to combining two factors at a time and giving a combination that produces the fewest intermediate variables based on the aforementioned common-factor rule. To be specific, we first carry out product operation for common two-factors $e^{\pm}_l \tilde{\Upsilon}_{il}, y_l (1-y_l), \mathcal{W}^{\pm}_{2_{1i}} (1-\tilde{\Upsilon_{il}})$ (shaded yellow in Fig.\ref{fig:SFP}), then combine single-factor $\Xi_{il}, e^{\pm}_l$ to perform second multiplication to obtain three-factor and four-factor shared monomials (shaded blue). After this step, we have the multiplication formula for $\mathcal{W}_{2}$ in the hidden-output connection. Finally, the formula for $\mathcal{W}_1$ is computed by the third multiplication. 

\begin{figure}[!t]
\centerline{\includegraphics[width=\columnwidth]{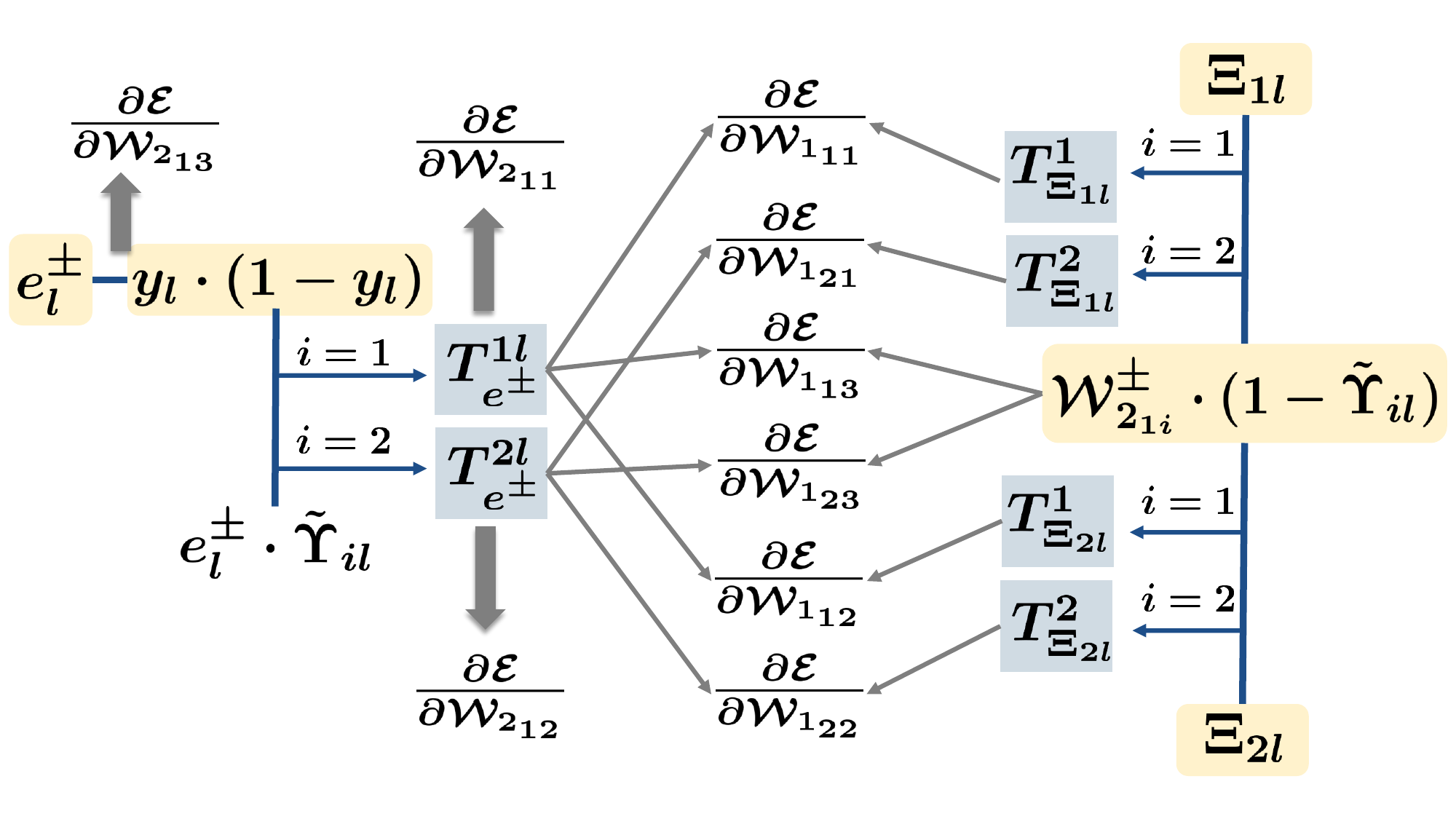}}
\caption{The computation flow diagram and intermediate variables produced of the SFP.}
\label{fig:SFP}
\end{figure}

Thus, given index sets $I_s = \{1,2,3,4,7,8\}$, $I_v = \{5,6,9\}$, the monomial species $\{Q^l_{s \pm \pm}\}_{s\in I_s}$, $\{Q^l_{v \pm}\}_{v \in I_v}$, the intermediate species $\{M^{a,\pm}_{l}, M^l_{y}, M^{l}_{b,\pm}\}$, $\{T^{a,\pm}_{l}, T^{l}_{b,\pm}\}$ with $a, b= \{1,2\}$ for $\forall l \in \{1,2,\cdots,\tilde{p}\}$ and gradient species pairs $\{Par^+_i, Par^-_i\}^9_{i=1}$ are introduced here, and then we exhibit the ngBCRN based on SFP by computing $-\frac{\partial \mathcal{E}}{\partial \mathcal{W}_{1_{11}}}$. Defining $par^{\pm}_1 \in \mathbb{R}^{2 \times 3}$, $par^{\pm}_2 \in \mathbb{R}^{1 \times 3}$ and using the \textit{dual-rail encoding}, we have
\begin{equation*}
    -\frac{\partial \mathcal{E}}{\partial  \mathcal{W}_{1_{11}}} = par^+_{1_{11}}- par^-_{1_{11}},
\end{equation*}
where 
\begin{gather*}
\begin{split}
    par^+_{1_{11}} & = \sum^{\tilde{p}}_{l=1} (e^+_l \mathcal{W}^+_{2_{11}} + e^-_l \mathcal{W}^-_{2_{11}}) y_l (1-y_l) \tilde{\Upsilon}_{1l} (1-\tilde{\Upsilon}_{1l}) \Xi_{1l} \\
par^-_{1_{11}} & = \sum^{\tilde{p}}_{l=1} (e^+_l \mathcal{W}^-_{2_{11}} + e^-_l \mathcal{W}^+_{2_{11}}) y_l (1-y_l) \tilde{\Upsilon}_{1l} (1-\tilde{\Upsilon}_{1l}) \Xi_{1l}.
    \end{split}
\end{gather*}
It involves computing $4\tilde{p}$ seven-factor monomials and requires species $\{Q^l_{1\pm\pm}\}$ and intermediate species of $a=1, b=1$. The multiplication computation when combining $e^+_l$ and $\mathcal{W}^+_{2_{11}}$ is implemented by designing the following MAS 
\begin{gather*}
\mathcal{O}^1_{27}:
\begin{split}
E^+_l+ P^l_{1} & \stackrel{1}{\longrightarrow} E^+_l + P^l_1 + M^{1,+}_{l}, ~~
M^{1,+}_{l} \stackrel{1}{\longrightarrow} \varnothing, \\
\tilde{Y}^l+ S^l_y & \stackrel{1}{\longrightarrow} \tilde{Y}^l+ S^l_y + M^{l}_{y}, ~~
M^{l}_{y} \stackrel{1}{\longrightarrow} \varnothing,\\
S^l_{p_1}+ W^+_5 &\stackrel{1}{\longrightarrow} S^l_{p_1}+ W^+_5 + M^{l}_{1,+},  ~~
M^{l}_{1,+} \stackrel{1}{\longrightarrow} \varnothing,\\
M^{1,+}_{l} + M^{l}_{y}& \stackrel{1}{\longrightarrow} M^{1,+}_{l} + M^{l}_{y} +T^{1,+}_{l},~~
T^{1,+}_{l}  \stackrel{1}{\longrightarrow} \varnothing,\\
M^{l}_{1,+} + S^l_1 & \stackrel{1}{\longrightarrow} M^{l}_{1,+} + S^l_1 + T^{l}_{1,+},~~
T^{l}_{1,+} \stackrel{1}{\longrightarrow} \varnothing,\\
T^{1,+}_{l} +T^{l}_{1,+} & \stackrel{1}{\longrightarrow} T^{1,+}_{l} +T^{l}_{1,+} + Q^l_{1 + +},~~ 
Q^l_{1++} \stackrel{1}{\longrightarrow} \varnothing
\end{split}
\end{gather*}
and the result is stored in $Q^l_{1++}$. A similar CRN design structure is applied in generating species $Q^l_{1+-}, Q^l_{1-+}, Q^l_{1--}$ with replacing the error species, weight species, and the intermediate species positions by $\{E^+, M^{1,+}_{l}, T^{1,+}_{l}, W^-_5, M^{l}_{1,-}, T^{l}_{1,-}\}$,  $\{E^-, M^{1,-}_{l}, T^{1,-}_{l}, W^+_5, M^{l}_{1,+}, T^{l}_{1,+}\}$ and  $\{E^-, M^{1,-}_{l}, T^{1,-}_{l},$ $W^-_5, M^{l}_{1,-}, T^{l}_{1,-}\}$, respectively. The CRNs in $\mathcal{O}^1_{27}$ for calculating negative gradient values concerning $\mathcal{W}_{1_{12}}$, $\mathcal{W}_{1_{21}}$, $\mathcal{W}_{1_{22}}$ have the same design as $\mathcal{W}_{1_{11}}$ but need the index $\{a,b\}$ of intermediate species set being $\{1,2\}$, $\{2,1\}$ and $\{2,2\}$, to finally get $\{Q^l_{3\pm\pm}\}$,  $\{Q^l_{2\pm\pm}\}$, $\{Q^l_{4\pm\pm}\}$,
respectively. Regarding $\mathcal{W}_{1_{13}}, \mathcal{W}_{1_{23}}$, we devise the new combination between $\{T^{1,\pm}_l\}$ and $\{M^l_{1,\pm}\}$, $\{T^{2,\pm}_l\}$ and $\{M^l_{2,\pm}\}$ to consequently produce $\{Q^l_{7\pm\pm}\}$ and $\{Q^l_{8\pm\pm}\}$, respectively. As for $\mathcal{W}_2$, it requires us to add only one new combination of species $E^{\pm}_l$ and $\{M^l_y\}$ to have new species $\{Q^l_{9\pm}\}$ that relates to $\mathcal{W}_{2_{13}}$ because the already existing $\{T^{1,\pm}_l\}$, $\{T^{2,\pm}_l\}$ can be regarded as $\{Q^l_{5\pm}\}$, $\{Q^l_{6\pm}\}$ for $\mathcal{W}_{2_{11}}$ and $\mathcal{W}_{2_{12}}$, separately. In the sum part, we devise reactions 
$\mathcal{O}^2_{27}$ below to calculate $par^{\pm}_1, par^{\pm}_2$
\begin{equation*}
\begin{array}{c|c}
\begin{aligned}
Q^l_{s++} & \stackrel{1}{\longrightarrow}
Q^l_{s++}  + Par^+_{s},\\
Q^l_{s--} & \stackrel{1}{\longrightarrow}
Q^l_{s--} + Par^+_{s},\\
Q^l_{v+} & \stackrel{1}{\longrightarrow}
Q^l_{v+} + Par^+_{v},\\
 Par^+_{s} & \stackrel{1}{\longrightarrow} \varnothing, Par^+_{v} \stackrel{1}{\longrightarrow} \varnothing,
\end{aligned}
&
\begin{aligned}
Q^l_{s+-} & \stackrel{1}{\longrightarrow}
Q^l_{s+-}  + Par^-_{s},\\
Q^l_{s-+} & \stackrel{1}{\longrightarrow}
Q^l_{s-+} + Par^-_{s},\\
Q^l_{v-} & \stackrel{1}{\longrightarrow}
Q^l_{v-} + Par^-_{v},\\
Par^-_{s} & \stackrel{1}{\longrightarrow} \varnothing, Par^-_{v} \stackrel{1}{\longrightarrow} \varnothing
\end{aligned}
\end{array}
\end{equation*}
with $s \in I_s$, $v\in I_v$. Following our design, all shared intermediate variables are calculated once, stored, and ready for arbitrary use, so the SFP-based design could be extended to the FCNN of any depth at a lower computational cost. 

Having known the negative gradient of $\mathcal{E}$ with respect to every weight, we consequently intend to get the change quantity $\Delta \mathcal{W}^m$ and the new weight $\mathcal{W}^{m+1}$ for the next iteration. The iterative formula \eqref{eq:GD} is rewritten as
\begin{gather}  \label{eq:new weight}
\begin{split}
     \mathcal{W}^{m+1}
     & = (\mathcal{W}^{m,+} + \eta \cdot par^{+})- (\mathcal{W}^{m,-} + \eta \cdot par^{-})\\
     & = \mathcal{W}^{m+1,+} - \mathcal{W}^{m+1,-}
\end{split}
\end{gather}
where $par^{\pm}=((par^{\pm}_1 )^{\top},(par^{\pm}_2)^{\top})^{\top}$.
To program \eqref{eq:new weight}, we add incremental weight species pairs $\{\Delta W^{+}_i, \Delta W^{-}_i\}^6_{i=1}$, incremental bias species pairs $\{\Delta B^{+}_j, \Delta B^{-}_j\}^3_{j=1}$, old weight-bias species $\{\Gamma^+_i, \Gamma^-_i\}^9_{i=1}$, and learning rate species $\{L\}$. Then, we design the uBCRN below by adopting the same pattern in the previous work \cite{FAN20227}
\begin{equation*}
\mathcal{O}_{29}:
\begin{split} \label{uBCR_1}
Par^{\pm}_i + L & \xrightarrow{1} Par^{\pm}_i + L +\Delta W^{\pm}_{i},~ 
\Delta W^{\pm}_{i} \xrightarrow{1} W^{\pm}_{i},\\
\Gamma^{\pm}_{i}& \xrightarrow{1} \Gamma^{\pm}_{i} + W^{\pm}_{i},~ 
W^{\pm}_i \xrightarrow{1} \varnothing,\\
Par^{\pm}_{j+6} + L & \xrightarrow{1} Par^{\pm}_{j+6} + L +\Delta B^{\pm}_{j},~ 
\Delta B^{\pm}_{j} \xrightarrow{1} B^{\pm}_{j},\\
\Gamma^{\pm}_{j+6}& \xrightarrow{1} \Gamma^{\pm}_{j+6} + B^{\pm}_{j},~ 
B^{\pm}_j \xrightarrow{1} \varnothing.
    \end{split}~~~~~~~~~
\end{equation*}
with $i \in \{1,\cdots,6\}$, and $j=\{1,2,3\}$.
\begin{remark}
The natural isolation device, the oscillator, helps us to make $\{\Gamma^{\pm}_i\}^9_{i=1}$ carry the weights and bias of the current iteration in a much simpler way than before \cite{FAN20227} by the following reactions added in phase $\mathcal{O}_7$
\begin{equation*}
\mathcal{O}^2_{7}:~~
\begin{array}{c|c}
    \begin{split}
W^{\pm}_{i} & \xrightarrow{1} \Gamma^{\pm}_{i} + W^{\pm}_{i},\\
 \Gamma^{\pm}_{i} & \xrightarrow{1} \varnothing,  
    \end{split}
         &
    \begin{split}
B^{\pm}_{j} & \xrightarrow{1} \Gamma^{\pm}_{j+6} + B^{\pm}_{j},\\
 \Gamma^{\pm}_{j+6} & \xrightarrow{1} \varnothing,
    \end{split}
    \end{array}~~
\end{equation*} 
\end{remark}
which transfers the initial concentrations of $\{W^{\pm}_i\}^6_{i=1}$ and $\{B^{\pm}_{j=1}\}^3_{j=1}$ of this period to $\{\Gamma^{\pm}_i\}^9_{i=1}$. Apart from that, we set $l(0) = \eta$ and give any positive initial concentration of incremental species, the uBCRN realizes the weight update process for $k$th training round.


\subsection{Clear-out Module}
Until now, we have shown all the core module designs of the biochemical neural network. However, they do not fully satisfy the needs of an automated implementation of FCNN since some given initial concentration conditions are required in each iteration. We divide them into two categories: 1) the initial concentration should be completely emptied; 2) the initial concentration should be any fixed positive value. Therefore, in our design, some species concentrations in the sigBCRN and pBCRN need to be revised in the final module. 
\subsubsection{Nonlinear CRN} The initial concentration of output species pair $\{P^+_{il}, P^-_{il}\}, \{Y^+_l, Y^-_l\}$ should be zero before every iteration to make sure that only one of each pair should be present with concentration $\frac{1}{2}$ before $\mathcal{O}_{13}$. However, in practice, the final equilibrium concentration of these species after each iteration, depending on the actual output results in FCNN, may be unpredictable for us, i.e., some of them might not be zero as intended. Therefore, we give the \textit{clear-out} BCRN in the final phase to consume all their concentrations 
\begin{equation*}
\mathcal{O}^1_{31}:~ P^+_{il} \xrightarrow{1} \varnothing,~P^-_{il} \xrightarrow{1} \varnothing,~Y^+_{l} \xrightarrow{1} \varnothing,~Y^-_{l} \xrightarrow{1} \varnothing. 
\end{equation*}
\subsubsection{Pre-calculationg CRN} All the species comprising the product complex $Y^l_e+Y^l_s+\tilde{Y}^l, P^l_{is}+\tilde{P}^l_{is}$ should be none in the reactor for target equilibrium concentration which depends on initial concentrations closely in our construction. Species $I^l_y, I^l_{p_i}$ should be constant one before every iteration for subtraction. But the practical situation is that $I^l_y, I^l_{p_i}$ must be consumed every iteration and the others may have residual concentrations. Thus, we give both the \textit{clear-out} BCRN and the \textit{compensatory} BCRN (far right)
\begin{equation*}
\mathcal{O}^2_{31}:
\begin{array}{c|c|c}
\begin{split}
    Y^l_{s} \xrightarrow{1} \varnothing, \\
    \tilde{Y}^l \xrightarrow{1} \varnothing, \\
    Y^l_{e} \xrightarrow{1} \varnothing, 
\end{split}
     &
\begin{split}
    P^{\pm}_{il} \xrightarrow{1} \varnothing, \\
    \tilde{P}^l_{i} \xrightarrow{1} \varnothing, 
\end{split}
&
\begin{split}
   I & \xrightarrow{1}I + I^l_{y}, \\
   I  & \xrightarrow{1}  I + I^l_{p_i},\\
I^l_{y} & \xrightarrow{1} \varnothing,~I^l_{p_i} \xrightarrow{1} \varnothing,
\end{split}
\end{array}
\end{equation*}
where $I$ is a unit species with the initial concentration one.

\section{Convergence Analysis of the Computation Reaction modules}
\label{sec:dynamic analysis}
Utilizing such a realization strategy, the asymptotic stability and convergence speed of the designed biochemical reaction systems are significant factors in determining the performance of biochemical neural networks. In this section, a rigorous stability analysis of the computational modules built above will be carried out to demonstrate that all computation modules have exponential convergence properties, which ensures that the BFCNN system can approach the final results at a very fast rate and has a convergent training process.

\subsection{Feedforward Propagation Module}
The asymptotic stability of the assignment module is shown in Proposition \ref{prop.1} and the lwsBCRN system has been shown to be a GAS system for given parameters. Here, we still study reactions of the sigBCRN system occurring in different phases separately. The following lemma is first given to show the exponential convergence of the reaction system in $\mathcal{O}_9$.
\begin{lemma} \label{lemma_O9}
Suppose that $n_{il}(t) =(n^+_{il}(t), n^-_{il}(t))^{\top}, c=n^-_{il}(0)-n^+_{il}(0), a = \frac{n^+_{il}(0)}{n^-_{il}(0)}$. There exists $ M_1 > \frac{c}{1-a}, \alpha_1 \in (0,-c)$ such that the MAS \eqref{CRNinO_9} converges exponentially to the unique boundary equilibrium $(\bar{n}_{il})^1 = (n^+_{il}(0)-n^-_{il}(0),0)$ in any compatibility class $\{(x,y) \in \mathbb{R}^2_{\geq 0}|y=x+b\}$ for $\forall b \in \mathbb{R}_{\geq 0}$, and also $\exists ~ M_2 > \frac{ac}{1-a}, \alpha_2 \in (0,c)$ such that it converges exponentially to the unique equilibrium $(\bar{n}_{il})^2 = (0,n^-_{il}(0)-n^+_{il}(0))$ in the compatibility class $\{(x,y) \in \mathbb{R}^2_{\geq 0}|y=x+b\}$ for $\forall b \in \mathbb{R}_{<0}$.
\end{lemma}

Lemma \ref{lemma_O9} shows that \eqref{CRNinO_9} realize computing the net input at a fast speed by regarding $N^+_{il}$ as output signal carriers when $n_{il} > 0$, and using species $N^-_{il}$ if $n_{il} < 0$ in FCNN. Using $p^{\pm}_{il}(t)$ to denote the concentration of species $P^{\pm}_{il}$, the dynamic equations of the MAS \eqref{eq:O_911}
\begin{equation}
\label{eq:O_1129}
\mathcal{O}_{11}:~~\dot{p}^{\pm}_{il}(t) = n^{\pm}_{il} (t) (\frac{1}{2}-p^{\pm}_{il}(t))      
\end{equation}
with $n^{\pm}_{il}(t) = n^{\pm}_{il}(0) \neq 0$ has the unique globally exponentially stable (GES) equilibrium point $\frac{1}{2}$ whose convergence rate depends on $\bar{n}^{\pm}_{il}$ in phase $\mathcal{O}_{9}$, that is, $|p^{\pm}_{il}(t) - \bar{p}^{\pm}_{il}|=|p^{\pm}_{il}(t)-\frac{1}{2}| = |p^{\pm}_{il}(0) - \frac{1}{2}|e^{-n^{\pm}_{il}(0)t}$. If one of $n^{\pm}_{il}(0)=0$, we can have $p^{\pm}_{il}(t) = p^{\pm}_{il}(0)$ for $\forall t \geq 0$. Then, the proposition below implies that MAS \eqref{eq:SIG_p}, \eqref{eq:SIG_n} of sgiBCRN will globally converge exponentially to an equilibrium set in which at least one species concentration maintains at zero, and provides the appropriate initial concentrations range for $P^{\pm}_{il}$ to guarantee that their equilibrium concentration always falls within the range of Sigmoid function.
\begin{proposition}\label{prop_O13}
For all $l = \{1,2,\cdots,\tilde{p}\}$, $i=\{1,2\}$, given nonnegative initial condition $(p^+_{il}(0), n^+_{il}(0)) \in \mathbb{R}^2_{\geq 0}$ and $(p^-_{il}(0), n^-_{il}(0))  \in \mathbb{R}^2_{\geq 0}$, the reaction system \eqref{eq:SIG_p} and \eqref{eq:SIG_n} will globally asymptotically converge to a unique boundary equilibrium set $\mathcal{A}_p = \{(p^+_{il}, n^+_{il})|p^+_{il} \geq 0, n^+_{il} = 0\}$ and $\mathcal{A}_n = \{(p^-_{il}, n^-_{il})|p^-_{il} \geq 0, n^-_{il} = 0\}$, respectively. Moreover, given initial concentrations $p^+_{il}(0), p^-_{il}(0) \in \left(0,1\right)$, there exist $M^l_p = n^+_{il}(0), M^l_n = n^-_{il}(0), \gamma_p = \gamma_n =1$ such that \eqref{eq:SIG_p} and \eqref{eq:SIG_n} can converge exponentially to $A^e_{p}=\{(p^+_{il}, n^+_{il})|p^+_{il} \in \left(0,1\right), n^+_{il} = 0\}$ and $A^e_{n}=\{(p^-_{il}, n^-_{il})|p^-_{il} \in \left(0,1\right), n^-_{il} = 0\}$, respectively.
\end{proposition}

Proposition \ref{prop_O13} provide a theoretical support that as long as the initial concentration of the corresponding output species is fixed at $\frac{1}{2}$, the MAS \eqref{eq:SIG_p} and \eqref{eq:SIG_n} have the limiting steady states of output species to represent the accurate Sigmoid activation output for arbitrary net input. Although \eqref{eqq:Sig_p}, \eqref{eqq:Sig_n} are both degenerate systems only admitting non-isolated boundary equilibrium points in the compatibility class, the attractive regions of every equilibrium can be determined from $\{(p^+_{il0}, n^+_{il0}) | p^+_{il0} = \frac{1}{1+\frac{1-b}{b} e ^ {n^l_{ip0}}}, b = \frac{1}{1+e^{-a}}, \forall a \in \mathbb{R}\}$ and $\{(p^-_{il0}, n^-_{il0}) | p^-_{il0} = \frac{1}{1+\frac{1-b}{b} e ^ {-n^l_{in0}}}, b = \frac{1}{1+e^{a}}, \forall a \in \mathbb{R}\}$.
 
Additionally, to save the running time of BFCNN by decreasing the number of phases, the MAS \eqref{eq:fO_13}, serving to represent target outputs by one species concentration, is embedded as a subnetwork into the MAS in phase $\mathcal{O}_{13}$. This strategy is also adopted in the subsequent design. We can demonstrate through the following lemma \ref{lemma_o13_F} that the asymptotic stability and the exponential convergence rate if any in the unembedded system, could be inherited in the parent reaction network.   
\begin{lemma} \label{lemma_o13_F}
Consider the kinetic equation of MAS $\mathscr{C}$ below 
\begin{equation}
    \dot{s}(t) = g(s(t)),~ ~s_0 =s(0),
\label{eq:lemma21}    
\end{equation}
where $s(t) \in \mathbb{R}^n_{\geq 0}$, and the set $I \subset \{1,2,...,n\}$. Assume that a certain species subset $\{S_i\}_{i \in I}$ of $\mathscr{C}$ has the limiting steady states and admits the exponential convergence, i.e., $\exists ~q_i, \beta_i > 0$ such that $\vert s_i(t) - \bar{s}_i \vert < q_i e^{- \beta_i t}$ for $\forall t \geq 0$. If the MAS $\mathscr{C}^{'}$ (called type-$\Rmnum{1}$ CRN)
\begin{gather}
    \begin{split}
 S_i & \stackrel{1}{\longrightarrow} S_i + H_, ~~ \forall i \in I, \\
   H & \stackrel{1}{\longrightarrow} \varnothing 
    \end{split}
\end{gather}
is embedded as a subsystem into $\mathscr{C}$, then species $H$ in the parent network $\mathscr{C}^{\ast}$ has the limiting steady state $\sum \bar{s}_i$ and will inherit the exponential convergence from $\mathscr{C}^{'}$. Moreover, if the equilibrium points of $S_i$ are asymptotically stable with respect to any finite initial concentrations (not necessarily exponentially) in $\mathscr{C}$, so as the $H$ in $\mathscr{C}^{\ast}$.
\end{lemma}

Lemma \ref{lemma_o13_F} helps to show that the MAS in phase $\mathcal{O}_{13}$ exponentially converge to the output result in hidden-layer neurons, which is stored in species $P^l_i$, even though there are no isolated LAS/GAS points for the MAS \eqref{eq:SIG_p} and \eqref{eq:SIG_n}. Then, we conclude with the following.
\begin{theorem} \label{thm1}
The biochemical feedforward module composed of the MAS \eqref{dy:n=w_1x}, \eqref{CRNinO_9}-\eqref{eq:fO_13} implements the basic feedforward computation with an exponential convergence to the target results. Such design can be extended to the feedforward FCNN of any depth and width.
\end{theorem}

\subsection{Preceding Calculation System}
The pBCRN is composed of \eqref{eq:bound1} - \eqref{eq:bound3}. For showing the asymptotic stability, we first consider the uppermost part \eqref{eq:bound1}, which can be decomposed into three independent subnetworks \cite{feinberg1979lectures} with disjoint species set $\mathscr{M}^1 = \{Y^l, Y^l_e, Y^l_s, \tilde{Y}^l, S^l_3, I^l_y\}$, $\mathscr{M}_{(i)}^2 = \{P^l_i, P^l_{is},\tilde{P}^l_{i}, I^l_{p_i}\}$ with $i=1,2$. For convenience, we represent $\mathscr{M}^1, \mathscr{M}^2$ using the following six-dimensional network and four-dimensional CRNs
\begin{equation}
\begin{array}{c|c}
\underbrace{\begin{split}
X_1 & \stackrel{k}{\longrightarrow} X_2 + X_3 +X_4 \\
X_2  + \tilde{X}_2 & \stackrel{k}{\longrightarrow} \varnothing \\
X_3  + \tilde{X}_3  & \stackrel{k}{\longrightarrow} \varnothing 
\end{split}}_{\mathscr{M}^1}
&
\underbrace{\begin{split}
Y_1 & \stackrel{k}{\longrightarrow} Y_2 +Y_3 \\
Y_2 + \tilde{Y}_2 &\stackrel{k}{\longrightarrow} \varnothing 
\end{split}}_{\mathscr{M}^2}
\end{array}.  
\end{equation}
Note that $\mathscr{M}^1$ and $\mathscr{M}^2$ have the same structure with one replication reaction and several annihilation reactions. The following section aims to prove that both have boundary equilibrium points that are GAS relative to the compatibility class.

\begin{proposition}\label{prop_bCRN1}
Given that $y=(y_1,y_2,y_3,\tilde{y}_2)^{\top}, y(0)=(y_1(0), y_2(0), y_3(0), \tilde{y}_2(0))^{\top} \in \mathbb{R}^{4}_{\geq 0}, d_1=y_1(0)+y_3(0), d_2=y_1(0)+y_2(0)-\tilde{y}_2(0)$. For any positive constant $k \in \mathbb{R}_{\geq 0}$ the MAS $\mathscr{M}^2$ with the dynamical equations
\begin{gather*}
\begin{split}
~\dot{y}_1(t) & = - k y_1(t), ~ \dot{y}_2(t) = k y_1(t) - k y_2(t)\tilde{y}_2(t), \\
\dot{y}_3(t) &  = k y_1(t),~\dot{\tilde{y}}_2=-k y_2(t) \tilde{y}_2(t)~~
\end{split}
\end{gather*}
is GAS in a stoichiometric compatibility class.  
\end{proposition}

The proof in the Appendix is finished by utilizing the candidate polyhedral Lyapunov function $V = \Vert \bm{y} - \bm{\bar{y}} \Vert _1$ for stability analysis. In particular, it is impossible to find a uniform positive constant $K$ for all $t \geq 0$ to describe the global exponential stability of $\mathscr{M}^2$ because multiple boundary equilibrium reactions in our design force the reaction rate slow when the trajectory is about to reach equilibrium. Therefore, we provide the following definitions.
\begin{definition}[Out-epsilon region]
For any point $\bar{x}$ in $\mathbb{R}^n_{\geq 0}$ and $\forall \epsilon > 0$, we call the region $\mathcal{O} =\{x \in \mathbb{R}^n_{\geq 0}| \Vert x -\bar{x} \Vert_1 \geq \epsilon \}$ \textit{out-epsilon region of $\bar{x}$}.
\end{definition}
\begin{definition}[Finite-time exponential stability]
Consider a MAS admitting an asymptotic stable equilibrium point $\bar{x}$. It is \textit{finite-time exponential stable} if for any finite time $T$, there exists an out-epsilon region $\mathcal{O}_T$ of $\bar{x}$ and positive constants $M,\gamma$ \textgreater $0$  such that $\Vert x(t) - \bar{x} \Vert \leq M e^{-\gamma t}$ for $x(t) \in \mathcal{O}_T$. 
\end{definition}

It implies that the system will converge exponentially to any $\epsilon$-neighbourhood of the equilibrium. Then, we give the convergence description of the MAS $\mathscr{M}^2$.
\begin{proposition}\label{prop_finit_ex}
The MAS $\mathscr{M}^2$ allows the finite-time exponential stability for the boundary equilibrium point in any stoichiometric compatibility class.
\end{proposition}

It is worth saying that the exponential order will decrease following the decrease of $\epsilon$, which is consistent with our intuition. The following proposition shows the stability of $\mathscr{M}^1$.
\begin{proposition} \label{prop_bCRN2}
Let $x(t)$ be the concentration vector of $\mathscr{M}^1$ and $\bm{x}(0) = (x_1(0),x_2(0),x_3(0),x_4(0),\tilde{x}_2(0), \tilde{x}_3(0))^{\top} \in \mathbb{R}^6_{\geq 0}$. Assume that $d_1=x_1(0)+x_4(0), d_2=x_1(0)-\tilde{x}_3(0)+x_3(0), d_3=x_1(0)-\tilde{x}_2(0)+x_2(0)$. For any positive constant $k$, the MAS $\mathscr{M}^1$ with equations 
\begin{gather*}
\begin{split}
    ~\dot{x}_1(t) & = - k x_1(t), ~\dot{x}_4(t) =  k x_1(t),\\
    \dot{x}_2(t) &= k x_1(t) - k x_2(t)\tilde{x}_2(t),~\dot{\tilde{x}}_2=-k x_2(t) \tilde{x}_2(t),\\
    \dot{x}_3(t) &= k x_1(t)-k x_3(t)\tilde{x}_3(t),~\dot{\tilde{x}}_3=-k x_3(t) \tilde{x}_3(t)~~
\end{split}
\end{gather*}
is GAS in a stoichiometric compatibility class and allows finite-time exponential stability.
\end{proposition}


Since $\mathscr{M}^1$ and $\mathscr{M}^2$ are two independent subnetworks without non-coupling components, the parent network \eqref{eq:bound1} also admits a unique global asymptotically stable boundary equilibrium and finite-time exponential stability following Propositions \ref{prop_bCRN1} - \ref{prop_bCRN2}. Further, let $x^l_e=(e^+_l,e^-_l,s^l_y)^{\top}, y^l_p=(s^l_{p1},s^l_{p2})^{\top}$. The exponential convergence of the pBCRN system in the out-epsilon region is consequently confirmed by considering $x^l_e, y^l_p, e^l$ via Lemma \ref{lemma_o13_F}. Regarding the jBCRN system, Lemma \ref{lemma_o13_F} also promotes to induce the asymptotic stability of species $C_a$ following the bistability of \eqref{eq:bistable}.


\subsection{Backpropagation Module}
\begin{figure}[!t]
\centerline{\includegraphics[width=\columnwidth]{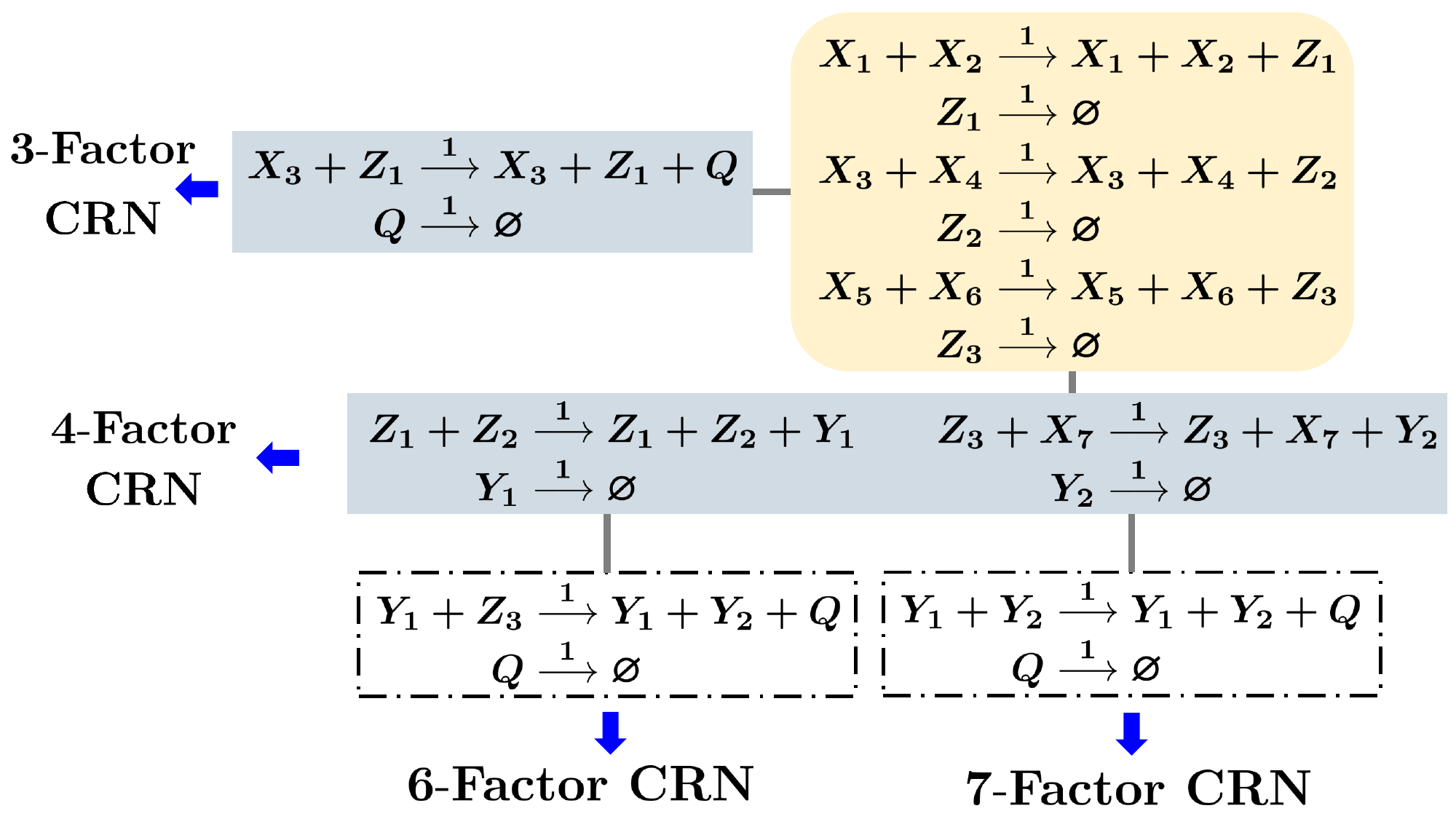}}
\caption{The schematic diagrams of reaction network structures appeared in $\mathcal{O}^1_{27}$. The \textit{root block} (shaded yellow) for the first product is shared by all factor CRNs. Different \textit{stem blocks} (shaded blue) for the second product are shared by different factor CRNs. Only six/seven-factor CRNs need the \textit{leaf blocks} (dashed frame) for the third multiplication.}
\label{fig:3}
\end{figure}
The MAS uBCRN has been shown to be GES in the past work \cite{FAN20227}. In this section, we work on the ngBCRN system and show that it has a GES equilibrium point, which means that all reactions occurring in parallel can perform sequential calculations to get all negative gradient values. Moreover, the final fixed point of the learning module after phase $\mathcal{O}_{29}$ is exactly the fixed point of the MBGD iterative algorithm.

In phase $\mathcal{O}_{27}$, four structural paradigms designed for computing multiplication are shown in Fig.\ref{fig:3}. Every paradigm is constituted by several blocks that are used to calculate two-factor multiplications. In particular, the concentration of reactant complexes of the root block is fixed, whereas the reactant complexes of stem and leaf blocks come from the product complexes in the upper blocks whose concentration might vary over time. Thus, we propose a variant of Lemma \ref{lemma_o13_F} to handle such the overlay of multiplication blocks.
\begin{lemma} \label{lemma_O27}
Consider the MAS $\mathscr{C}$ with the kinetic equation \eqref{eq:lemma21}, and assume that the certain species subset $\{S_1, S_2\}$ has the limiting steady states, and $\exists ~q_i, \beta_i > 0$ such that $\vert s_i(t) - \bar{s}_i \vert < q_i e^{- \beta_i t}$ for $\forall t \geq 0$. Then if we embed the MAS $\tilde{\mathscr{C}}$ (called type-\Rmnum{2} CRN) 
\begin{gather}
    \begin{split}
 S_1 +S_2 & \stackrel{1}{\longrightarrow} S_1 + S_2 +H, \\
    H &\stackrel{1}{\longrightarrow} \varnothing 
    \end{split}
\end{gather}
into $\mathscr{C}$, the species $H$ will inherit the exponential convergence of equilibrium $\bar{s}_1 \bar{s}_2$ in the parent network $\mathscr{C}^*$.
\end{lemma}

Since the reaction system in the root block in Fig.\ref{fig:3} is GAS, all types of factor CRNs systems are easily shown to be GAS in a similar way as the proof of Lemma \ref{lemma_o13_F}. Then combining Lemma \ref{lemma_O27}, we have the next theorem.
\begin{theorem} \label{thm_final}
Given any nonnegative catalyst concentration, ngBCRN allows global exponential stability and implements a negative gradient calculation. Moreover, the fixed point of uBCRN equals the new weight matrix in every iteration.
\end{theorem}

Therefore, based on our termination mechanism, the final fixed point of the uBCR system is exactly the fixed point of the mini-batch gradient descent algorithm. Additionally, compared to \cite{lakin2023design}, the computation results of our backpropagation module are robust to the initial concentrations of gradient species and weights species due to GES.
\section{Application to binary classification of logic gates problems} 
\label{sec:case study}
Following the construction proposed above, in this section, we will test the learning performance of our BFCNN using two typical logic classification problems. One is a simple linear separable case ``OR'' gate with the training sample $\chi^1_{\cdot 1}=(1,0,1)^{\top},\chi^1_{\cdot 2}=(0,0,0)^{\top},\chi^1_{\cdot 3}=(1,1,1)^{\top},\chi^1_{\cdot 4}=(0,1,1)^{\top}$, and another is a linear inseparable case ``XOR'' with $\chi^2_{\cdot 1}=(1,0,1)^{\top},\chi^2_{\cdot 2}=(0,0,0)^{\top},\chi^2_{\cdot 3}=(1,1,0)^{\top},\chi^2_{\cdot 4}=(0,1,1)^{\top}$.
The entire run of BFCNN occupies $29$ phases of chemical and we give $32$ phases with $32$ species for buffering purposes. In our simulations, the rate constant $k_o$ is given as $2$, the initial concentrations for the last two species are set to $1$ and others are $10^{-6}$ \cite{arredondo2022supervised}. In order to train the biological neural network, we begin with two initial weights species concentrations (displayed in Appendix \Rmnum{1}), set the error threshold to be $0.5$ using rate constants $k_1 = 8, k_2 =1, k_3=2, k_4=0.4375$ in \eqref{eq:bistable}, let the learning rate $\eta=0.9$ for both cases and force the batch size to be $\tilde{p}=2$ for the four training samples. These settings ensure that the training was terminated in the $12_{th}$ iteration for ``XOR'' and the $5_{th}$ training round for ``OR'' with the final weights (shown in Appendix \Rmnum{1}) and actual outputs $y_{OR} = [0.6654, 0.3950,0.6358,0.5724]$ and $y_{XOR}=[0.5129,0.4188.0.4503,0.5082]$. To visualize the linear and nonlinear classification surface, we generate $100$ points in the plane $(0,1) \times (0,1)$ and feed them into the feedforward part with optimal weights in BFCNN to compute corresponding outputs and finally, we get the heatmap shown in Fig.\ref{fig:9}.

\begin{figure}[!t]
\centerline{\includegraphics[width=\columnwidth]{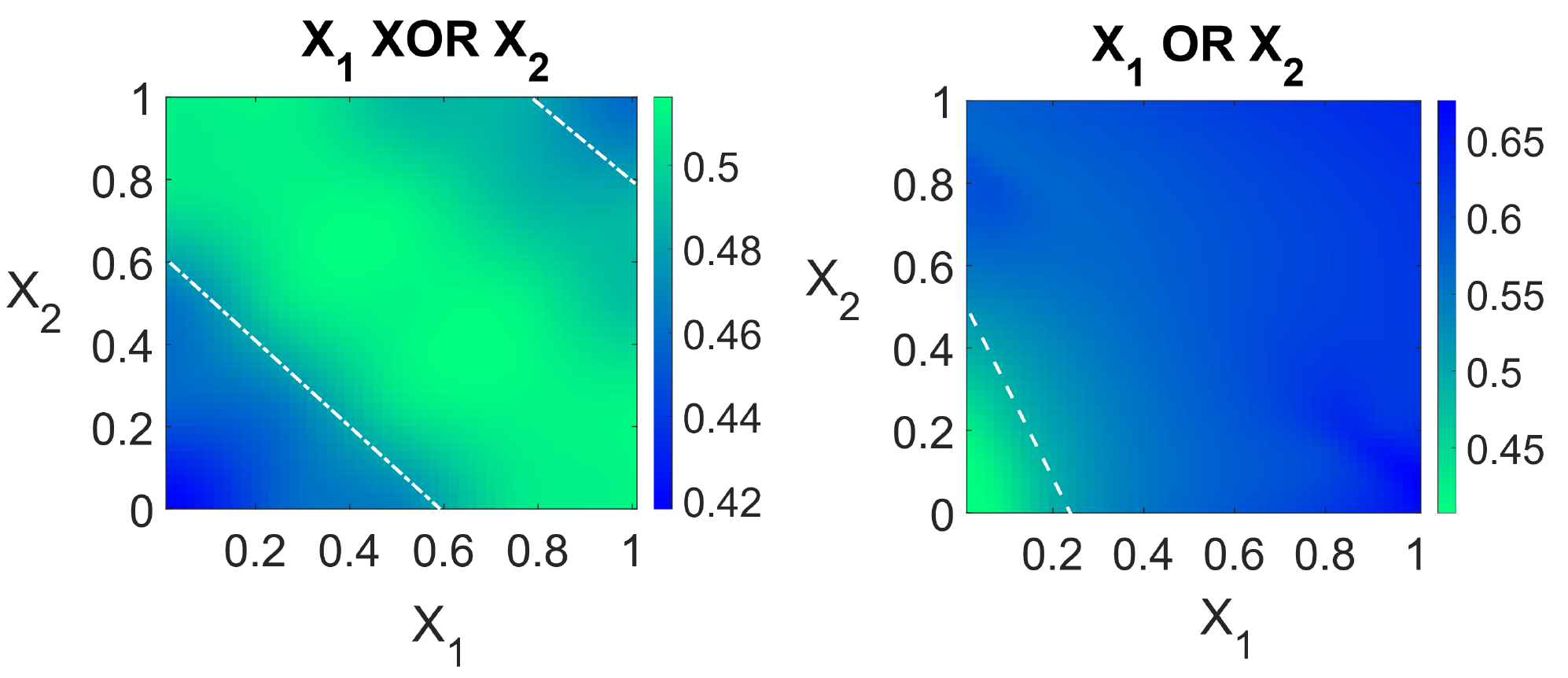}}
\caption{The heatmaps show the approximated decision surface (dotted white line) for the binary classification functions:``$X_1$ XOR $X_2$'' (left) and ``$X_1$ OR $X_2$'' (right).}
\label{fig:9}
\end{figure}

\section{Conclusion}
In Part \Rmnum{1} of this article, we design a biochemical reaction system capable of learning that first implements a complete fully connected neural network with the gradient descent based backpropagation process. The assignment module is an innovative mechanism to automatically input training samples like computers do, which can be applied to the biological realization of any iterative algorithm. The Sigmoid function with respect to real inputs is precisely computed by a revision of the autocatalytic reactions. Also, to equip our system with the learning capacity, we give an accurate and robust realization of the learning module whose final equilibrium concentration is exactly the fixed point of the gradient descent algorithm. Then, we provide the theoretical support for our design through a rigorous dynamic analysis to ensure asymptotic stability and exponential convergence, which might facilitate the building of a computational module library with explicit properties benefiting subsequent uses. 

Even though all computational modules allow exponential convergence, the realization errors cannot be avoided due to the finite phase length of oscillators and the complicated construction. In Part \Rmnum{2}, we propose a strict error analysis method and give an error upper bound of our BFCNN. In addition, all designs rely on precise rate constant settings and certain initial concentration choices, while uncertainties, noises, and disturbances always exist when it comes to biological environments. Therefore, enhancing the robustness of computational modules combined with the control theory will be considered in future work.

\section*{Appendix}
\appendices
In this appendix, we provide all proofs to the results given in the previous sections. 

\noindent \textbf{Proof of Proposition \ref{prop.1}:}

\begin{proof}
The oscillatory signals $O_1$, $O_3$ and $O_5$ will drive the designed 3 phases to occur orderly, so we can discuss their dynamics separately. Also, note that there are species coupling among 3 phases, so the given initial concentrations for some species are not global, such as for $C_i^l$ and $\tilde{C}^l_i$, which will change to be the limiting steady state points in the last phase. 

Firstly, for the phase $\mathcal{O}_1$, $\forall l \in \{1,...,\tilde{p}\}$ the dynamics takes
\begin{equation}
    \dot{s}^l(t) =\mathcal{X}(t) ({c^l} (t))^\top - s^l(t),
    \label{dy:O1}
\end{equation}
where $\mathcal{X}\in \mathbb{R}^{3\times p}$ satisfies $\mathcal{X}_{.i}(t)=(x^i_1(t),x^i_2(t),d^i(t))^\top$. Note that we give a concise form of (\ref{dy:O1}) to express dynamics, which does not follow the form of (\ref{eq:2}) but is the same with the latter essentially. From (\ref{eq:O1}), since species $X^i_q,~C_i^l,~D^i$ are all catalysts, we have $\mathcal{X}(0)=\mathcal{X}(t),~ c^l(0)=c^l(t),~ \forall t \geq 0$. The solution of (\ref{dy:O1}) is thus easily got $$s^l(t)=\mathcal{X}(0)(c^l(0))^\top(1-e^{-t})+s^l(0)e^{-t},$$ which implies that $\mathscr{L}[{s}^l(t)]=\bar{s}^l=\mathcal{X}(0)(c^l(0))^\top$ is GAS $\forall l$. We then realize to assign $block_1=(\bar{s}^1,...,\bar{s}^{\tilde p})=(\chi_{.1},...,\chi_{.\tilde p})$ to the input layer utilizing $\mathcal{O}_1$.  


Next, for the phase $\mathcal{O}_3$ it is easy from its dynamics to get $\mathscr{L}[c(t)]=\mathbb{0}$ and $\mathscr{L}[\tilde{c}(t)]=c(0)+\tilde{c}(0)=c(0)$, which is a GAS boundary equilibrium point relative to its stoichiometry compatibility class of any given initial point. This phase serves to use species $\tilde{C}_i^l$ to store the initial concentrations of $C_i^l$, but clear the latter's concentrations at the same time. 

Last, for the phase $\mathcal{O}_5$, $\forall l \in \{1,...,\tilde{p}$\} the dynamics follows
\begin{gather} \label{dy:o5}
\begin{split}
    \dot{c}^l_{l}(t) & = k \tilde{c}^l_{p-\tilde{p}+l}(t), ~\dot{c}^l_{o+\tilde{p}}(t) = k \tilde{c}^l_{o}(t), \\
    \dot{c}^l_j(t) & = k \tilde{c}^l_j(t),~~
    \dot{\tilde{c}}^l_i(t)= -k \tilde{c}^l_i(t),
\end{split}
\end{gather}
where $o \in I_l$, $j \notin I_l \bigcup \{p-\tilde{p}+l\}$ and $i \in \{1,...,p\}$. Similarly, the MAS of \eqref{dy:o5} is GAS in each stoichiometric compatibility class and has the limiting steady states 
\begin{gather}\label{eq:o5limi}
\begin{split}
    \mathscr{L}[c^l_{o+\tilde{p}}(t)] & =\tilde{c}^l_{o}(0),~
    \mathscr{L}[c^l_l(t)]  =\tilde{c}^l_{p-\tilde{p}+l}(0),\\
   \mathscr{L}[c^l_j(t)] &=\tilde{c}^l_j(0),~   \mathscr{L}[\tilde{c}(t)] = \mathbb{0}.
\end{split}
\end{gather}
This phase serves to update the equilibrium concentrations $\bar{c}$ of $C_i^l$ based on the real-time concentrations of $\tilde{C}_i^l$, and meanwhile clear the concentrations of $\tilde{C}_i^l$. Note that this update process is not a one-to-one assignment between $\bar{c}_i^l$ and $\tilde{c}^l_i(0)$ according to $i$ and $l$, but more like a ``mismatch'' assignment. As a result, $\bar{c}$ will change to be $\bar{c}=(e_{\tilde{p}+1},...,e_{2\tilde{p}})^\top$. Then the phase $\mathcal{O}_5$ shuts down while the phase $\mathcal{O}_1$ starts up through the evolution of the oscillatory signals $O_5$ and $O_1$. The limiting steady state $\bar{c}=(e_{\tilde{p}+1},...,e_{2\tilde{p}})^\top$ will act as the initial point of $c(t)$ to drive the phase $\mathcal{O}_1$ to start the new round of computation, we then get $block_2=(\chi_{.{\tilde{p}+1}},...,\chi_{.2\tilde p})$ through (\ref{dy:O1}) and realize the new round of assignment to the input layer for the next iteration. As the oscillatory signals $O_1$, $O_3$ and $O_5$ go on, $\bar{c}^l$ will change periodically as 
\begin{equation*}
    \underbrace{e_l}_{\bar{c}^l_{m=1}} \rightarrow \underbrace{e_{\tilde{p}+l}}_{\bar{c}^l_{m=2}} \rightarrow \cdots \rightarrow \underbrace{e_{(\frac{p}{\tilde{p}}-1)p+l}}_{\bar{c}^l_{m=p-\tilde{p}+l}} \rightarrow \underbrace{e_l}_{\bar{c}^l_{m=1+\frac{p}{\tilde{p}}}} \rightarrow \cdots,
\end{equation*}
and the continuous assignment of feeding 
$``block_1 \rightarrow block_2 \rightarrow \cdots \rightarrow block_{\frac{p}{\tilde{p}}} \rightarrow block_1 \rightarrow \cdots"$ into the input layer one by one during each iteration is realized automatically.

\end{proof}

\noindent \textbf{Proof of Lemma \ref{lemma_O9}:}

\begin{proof}
For all $l$, the dynamical equations of \eqref{CRNinO_9} can be described as 
\begin{gather} \label{eq:O_928}
\mathcal{O}_9:~~\dot{n}^+_{il} (t) = \dot{n}^-_{il}(t) = -n^+_{il}(t) n^-_{il}(t). 
\end{gather}
The compatibility class of \eqref{CRNinO_9} is a lay line $l_0=\{(n^-_{il}(t)-n^+_{il}(t) = n^-_{il}(0)-n^+_{il}(0)) \cap \mathbb{R}^{2}_{\geq 0} | n_{il}(0) \in \mathbb{R}^2_{\geq 0}\}$ in the plane for any initial condition pair, and it has a unique boundary equilibrium in every invariant set $l_0$. The case of $n^+_{il}(0) > n^-_{il}(0)$ implies $c < 0, a > 1$, and the trajectory entirely falls in $\{(x,y) \in \mathbb{R}^2_{\geq 0}|y=x+n^+_{il}(0)-n^-_{il}(0)\}$. Deduced from the analytic solution to \eqref{eq:O_928}, we have $\Vert n_{il}(t) - \bar{n}_{il} \Vert = |n^{\pm}_{il}(t) - \bar{n}^{\pm}_{il}|=|\frac{c}{1-ae^{-ct}}|$. Let $g(t) = \tilde{M} e^{-(\alpha +c) t} - \tilde{M} e^{-\alpha t} - 1$. In order to ensure that $g(t) > 0 $ holds for $\forall t \geq 0$, we select $\tilde{M}_1 > \frac{1}{a-1}, 0< \alpha_1 < -c $ to make $g(0) = 0, g'(t) > 0$ for $\forall t \geq 0$. Hence, we get the following 
\begin{equation}\label{eq:lemma1_2}
    \tilde{M}_1e^{-\alpha t}-\frac{1}{ae^{-ct}-1} = \frac{g(t)}{ae^{-ct}-1}  > 0,~~ \forall t \geq 0.
\end{equation}
Let $M_1 > \frac{c}{1-a}$, \eqref{eq:lemma1_2} implies that $\Vert n_{il}(t) - (\bar{n}_{il})^1 \Vert < M_1e^{-\alpha t}$. Due to the symmetry of the two concentration variables, when $n^+_{il}(0) < n^-_{il}(0)$, we can find $\Vert n_{il}(t) - (\bar{n}_{il})^1 \Vert < M_2 e ^{-\alpha_2 t}$ where $M_2 > \frac{c}{\frac{1}{a}-1} = \frac{ca}{1-a}, \alpha_2 \in (0,c)$. Thus, the MAS \eqref{CRNinO_9} converge globally to the unique boundary equilibrium point in any compatibility class at an exponential rate.    
\end{proof}

\noindent \textbf{Proof of Proposition \ref{prop_O13}:}

\begin{proof}
Consider the dynamical equations of \eqref{eq:SIG_p} and \eqref{eq:SIG_n} as follows
\begin{align}\label{eqq:Sig_p}
\dot{p_{il}^+}(t) &= n_{il}^+(t)p^+_{il}(t)(1-p^+_{il}(t)), ~\dot{n_{il}^+}(t) = -n^+_{il}(t),\\ 
\dot{p_{il}^-}(t) & = n_{il}^-(t)p^-_{il}(t)(p^-_{il}(t)-1),~\dot{n_{il}^-}(t) = -n^-_{il}(t). \label{eqq:Sig_n}
\end{align}
Obviously, the stoichiometric subspace is the whole nonnegative orthant, and $\mathcal{A}_p$, $\mathcal{A}_n$ is the boundary equilibrium set of \eqref{eqq:Sig_p} and \eqref{eqq:Sig_n}, respectively. Let $z_p^l(t)=(p^+_{il}(t),n^+_{il}(t))^{\top}, z_n^l(t)=(p^-_{il}(t),n^-_{il}(t))^{\top}$. The analytic solution to \eqref{eqq:Sig_p} with any $(p^+_{il}(0),n^+_{il}(0)) \in \mathbb{R}^2_{> 0}$ (the case of boundary initial conditions is trivial) is 
\begin{gather*}
p^+_{il}(t) = \frac{1}{1+\frac{1-p^+_{il}(0)}{p^+_{il}(0)} e^{(e^{-t}-1) n^+_{il}(0)}},~~
n^+_{il}(t) = n^+_{il}(0) e^{-t}.
\end{gather*}
Let $C_1=\frac{1-p^+_{il}(0)}{p^+_{il}(0)} > -1, C_2=n^+_{il}(0)>0$. Thus, for any $\epsilon_1, \epsilon_2$ \textgreater $0$, there exist $T$ \textgreater $0$, when $t$ \textgreater $T$, we have 
\begin{equation}
    |p^+_{il}(t) - \frac{1}{1+C_1 e^{-C_2}}| \leq \epsilon_1, ~~|n^+_{il}| \leq \epsilon_2.
\label{eq:32}
\end{equation}
With $\epsilon=\max\{\epsilon_1, \epsilon_2\}$, \eqref{eq:32} means that $\Vert z_p^l(t) - \bar{z}_p^l \Vert $ \textless $\epsilon$, where $\bar{z}_p^l=(\frac{1}{1+C_1 e^{-C_2}}, 0)^{\top}\in \mathcal{A}_p$. In other words, $\mathcal{A}_p$ is consistent with the $\omega$-limit set of MAS \eqref{eqq:Sig_p}. So there exist $T>0$ for $\forall \epsilon > 0$ such that dist$(z^l_p(t), \mathcal{A}_p) = \inf_{y \in \mathcal{A}_p} \Vert z^l_p(t) - y \Vert \leq \Vert z_p^l(t) - \bar{z}_p^l \Vert$ \textless $\epsilon$, which implies $\mathcal{A}_p$ is GAS, i.e., $z^l_p(t) \rightarrow \mathcal{A}_p$ from any nonnegtive initial conditions. 

Now we claim that the convergence is exponential when initial species $P^+_{il}$ concentrations are smaller than one, i.e., there exists $M_p, \gamma_p > 0$ for which dist$(z^l_p(t), \mathcal{A}_p) < M_p e^{-\gamma_p t}$. Let $C=1+C_1e^{-C_2}$, so $C_1, C_2, C > 0$ due to $p_{ip}^l(0)>0$. Since $C_1e^{(e^{-t}-1)C_2}>C_1e^{-C_2}>0$, we have
\begin{equation*}
    \frac{1+C_1e^{-C_2}}{1+C_1e^{(e^{-t}-1)C_2}} > \frac{C_1e^{-C_2}}{C_1e^{(e^{-t}-1)C_2}}=e^{-C_2e^{-t}}.
\end{equation*}
In addition, $e^{-C_2e^{-t}} \geq -C_2e^{-t} +1$ always holds for $\forall t > 0$ owing to $e^t \geq t+1 $. Then, according to the inequality relationship above, one can derive that
\begin{equation*}
\begin{split}
   \left| p^+_{il}(t) - \frac{1}{1+C_1e^{-C_2}}  \right| & = \frac{1}{1+C_1e^{-C_2}} - \frac{1}{1+C_1e^{(e^{-t}-1)C_2}} \\
   & = \frac{1}{1+C_1e^{-C_2}} (1-\frac{1+C_1e^{-C_2}}{1+C_1e^{(e^{-t}-1)C_2}})\\
   & \leq \frac{C_2e^{-t}}{1+C_1e^{-C_2}}.
\end{split}
\end{equation*}
And let $\gamma_1 = 1, M^{'}_p > \frac{C_2}{C}$, we obtain $\left| p^+_{il}(t) - \frac{1}{1+C_1e^{-C_2}}  \right|$ $< M^{'}_p e^{-\gamma_1 t}$ for all $t \geq 0$. Furthermore, $C_1 > 0$ means that $\bar{p}^+_{il} \in (0,1)$ and so $\bar{z}^l_p \in \mathcal{A}^e_p$. Combining with $\left|n^+_{il}(t)\right|=C_2e^{-t}$ and setting $\gamma_p=1, M_p > \max\{C_2, \frac{C_2}{C}\}=C_2$, it holds that dist$(z^l_p(t), \mathcal{A}^e_p) \leq \Vert z^l_{ip}(t) - \bar{z}^l_P \Vert <  M_p e^{-\gamma_p t}$. A similar argument following the solution to the IVP of \eqref{eqq:Sig_n}  
\begin{gather*}
p^-_{1l}(t) = \frac{1}{1+\frac{1-p^-_{1l}(0)}{p^-_{1l}(0)} e^{(1-e^{-t}) n^-_{1l}(0)}},~~
n^-_{1l}(t) = n^-_{1l}(0) e^{-t}
\end{gather*}
shows that $A_n$ is also globally asymptotically stable and dist$(z^l_n(t), \mathcal{A}^e_n) <  M_n e^{-\gamma_n t}$ with $\gamma_n=1, M_n > \max \{n^-_{il}(0), p^-_{il}(0)n^-_{il}(0)\} = n^-_{il}(0)$.
\end{proof}

\noindent \textbf{Proof of Lemma \ref{lemma_o13_F}:}

\begin{proof}
The dynamical equations of $\mathscr{C}^{\ast}$ can be written as 
\begin{gather}
\begin{split}
    \dot{s}(t) & = g(s(t)),\\
    \dot{h}(t) & = \sum_{i \in I} s_i(t) - h(t). 
\end{split}
\end{gather}
Define $r(t) \triangleq \sum_{i \in I} s_i(t)$ and $ \bar{r} = \sum_{i \in I} \bar{s}_i$.
Then, the trajectory of species $H$ is $h(t) = e^{-t} (\int_{0}^{t} r(\tau) e^{\tau} \mathrm{d}\tau +h(0))$. The index set $I$ is partitioned into $I_1, I_2, I_3$ according to the size of the 
convergence order $\beta_i$ of species $S_i$, where $\beta_i \in (0,1)$ if $i \in I_1$, $\beta_j =1 $ if $j \in I_2$, and $\beta_k \in (1, \infty)$ if  $k \in I_3$. Then, we have
\begin{gather*}
    \begin{split}
\vert h(t) - \bar{r} \vert & = \vert e^{-t} \int_{0}^{t} r(\tau) e^{\tau} \mathrm{d}\tau +e^{-t}h(0) -\bar{r} \vert \\
& = \vert e^{-t} \int_{0}^{t} (r(\tau) - \bar{r}) e^{\tau} \mathrm{d}\tau + (h(0)-\bar{r}) e^{-t}\vert \\
& \leq e^{-t}  \int_{0}^{t} \sum_{i \in I} q_i e^{\tau -\beta_i \tau} \mathrm{d} \tau + \vert h(0)-\bar{r} \vert e^{-t} \\
& = \sum_{i \in I_1} a_i(e^{-\beta_i t} - e^{-t}) +\sum_{j\in I_3} b_j (e^{-\beta_j t} - e^{-t}) \\ 
& + \sum_{k \in I_2}q_k te^{-t} + G e^{-t}
    \end{split}
\end{gather*}
where $a_i = \frac{q_i}{1-\beta_i} > 0, i \in I_1$, $b_j = \frac{q_j}{1-\beta_j}<0, j \in I_3$, and $G= \vert h(0) - \bar{r} \vert$. Since $0<\beta_i < 1, \forall i \in I_1$ and let $\beta = \min_{i \in I_1} \{\beta_i\}$, the above inequality become
\begin{gather*}
    \begin{split}
\vert h(t)-\bar{r} \vert < (\sum_{i \in I_1} a_i -  \sum_{j\in I_3} b_j + G) e^{-\beta t} + \sum_{k \in I_2}q_k te^{-t}.
    \end{split}
\end{gather*}
For $te^{-t}$, there exists $\gamma \in (0,1-\frac{1}{e}), R=1$ such that $te^{-t} < Re^{-\gamma t}$. (In effect, $0 < \gamma < 1-\frac{1}{eR}$ ). Therefore, it shows that
$\vert h(t)-\bar{r} \vert <  Ke^{-a t}$ with $K= (\sum_{i \in I_1} a_i -  \sum_{j\in I_3} b_j +\sum_{k \in I_2}q_k+G) >0, a = \min\{\beta, \gamma\}$.

Let $d(t) = h(t)-\bar{r}$, we find that $\dot{d}(t) = -kd(t) +kf(t)$, where $f(t) = r(t)-\bar{r}$. Due to the globally asymptotic stability of $S_i$, for any sufficiently small $\epsilon >0 $, there exists a fixed time $T>0$, such that when $t>T$, $\vert f(t) \vert < \epsilon$. In the case of $t \in \left[T,\infty\right)$, we write the solution of $d(t)$ as 
\begin{equation*}
\begin{split}
 \vert d(t) \vert & \leq \vert d(T) \vert e^{k(T-t)} + e^{-kt} \int_{T}^{t} \vert f(s) \vert e^{ks}\mathrm{d}s\\
&  < \vert d(T) \vert e^{k(T-t)} + \frac{\epsilon}{k} +\frac{e^{k(T-t)}}{k}. 
\end{split}
\end{equation*}
Therefore, for the fixed $T$ and any $\epsilon_1 > 0$, there exist $T_1 >T$, such that $\vert d(t) \vert < \vert d(T)\vert \epsilon_1 + \frac{\epsilon}{k}+\frac{\epsilon_1}{k}$ when $t>T_1$. Since the state variable $f(t)$ is bound due to the GAS of $S_i$, any finite initial condition $d(0)$ can be chosen to ensure the boundedness of $d(t)$, i.e., $\vert d(t) \vert \leq M$. Thus, we have $\vert d(t) \vert < \delta(\epsilon, \epsilon_1) $ when $t > T_1$ with $\delta = M\epsilon_1+\frac{\epsilon}{k}+\frac{\epsilon_1}{k}$, and the GAS of species $H$ can be guaranteed.
\end{proof}

\noindent \textbf{Proof of Theorem \ref{thm1}:}

\begin{proof}
Denoting state variables in phases from $\mathcal{O}_7 - \mathcal{O}_{13}$ as four vectors of dimension-suited $x_i (i=1,2,3,4)$ (may include components for repetitive species), and representing the corresponding vector fields as $f_i$, the complete system is
\begin{gather*}
\begin{split}
\dot{x}_1 & = f_1(x_1) o_{7}(t),~\dot{x}_2 = f_2(x_2) o_{9}(t),\\
\dot{x}_3 & = f_3(x_3) o_{11}(t),~\dot{x}_4 = f_4(x_4) o_{13}(t),
\end{split}
\end{gather*}
where $o_{i}(t)$ denotes concentrations of oscillatory species. The oscillator characteristic ensures that changes in species concentration in different phases only follow the governing equations designed at that stage. Here, concentration change of $O_i$ is regarded as functions to $t$, and there exists $b_o > 0$ s.t. $0 \leq o_i(t) < b_o$ for the oscillator we used. Note that positive $o_{i}(t)$ will not change the equilibrium points of module systems, and their asymptotic stability can be guaranteed by the \textit{Converse Lyapunov Theorem} and Lemma \ref{lemma_O9}, Proposition \ref{prop_O13}, and Lemma \ref{lemma_o13_F} since $\frac{d V(x) }{d t} = \nabla V(x) f_i(x) o_i(t) < 0$. Also, if we consider the time interval when $o_i(t) \geq \tilde{b}_o$ for some $\tilde{b}_o >0$, the exponential convergence is still maintained. 

In this way, with the equilibrium concentration of $N^+_{il}, N^-_{il}$ of lwsBCRN being the initial concentration of the sigBCRN,  the output is obtained in $P_{il}$ after phase $\mathcal{O}_{13}$. To increase the realization width, we only need to introduce more species and reactions in the relevant phases, and to expand the depth, the number of phases with the same design should be added. All operations only induce the embedding of independent subsystems. 
\end{proof}

\noindent \textbf{Proof of Proposition \ref{prop_bCRN1}:}

\begin{proof}
For the given initial condition $y(0)$,  the trajectory of $\mathscr{M}^2$ falls entirely on the stoichiometric compatibility class $\mathcal{S}_1 =\{y \in \mathbb{R}^4_{\geq 0}|2y_1+y_2+y_3-\tilde{y}_2-d =0\}$ with $d=d_1+d_2$. It's easy to follow that $\bar{y}_1 = 0$ and $\bar{y}_3 = y_1(0)$ and $\bar{y}_2$, $\bar{\tilde{y}}_2$ rely on the sign of $d_2$. The case $d_2 > 0$ means that $\tilde{Y}_2$ will be first consumed and $\bar{y}=(0,d_2,d_1,0)^{\top}$. Note that $\dot{y}_3(t)>0, \dot{y}_1(t)<0, \dot{\tilde{y}}_2 < 0$, so $y_3(t) < d_1, y_1(t) < y_1(0), \tilde{y}_2 < \tilde{y}_2(0)$ for all $t \in \left[0, \infty \right)$. We have $V(y)=y_1 + \vert y_2 - d_2 \vert + d_1-y_3  + \tilde{y}_2$, whose behavior is examined in different regions below.
\begin{enumerate}
\item In $D^1 = \{y \in \mathcal{S}_1|y_1<y_1(0), y_2 \geq d_2, y_3(0) < y_3 \leq d_1, \tilde{y}_2<\tilde{y}_2(0) \}$, the Lyapunov derivative is 
\begin{equation*}
       D^+ V(y) = \dot{y}_1 +\dot{y}_2 -\dot{y}_3 + \dot{\tilde{y}}_2 = -ky_1-2ky_2\tilde{y}_2. 
\end{equation*}
Thus, $D^+V(y) < 0$ in the interior $\mathbb{R}^4_{>0}$. Also, the boundary region $D^1_a = \{y_1=0, y_2=0\} \not\subset D^1 $, so the trajectory $y(t)$ will not run to $D^1_a$ at any time, and in effect, $D^1_b = \{y \in D^1|y_1=0, \tilde{y}_2=0\}$ represent the boundary equilibrium point on $\mathcal{S}_1$. Therefore, $D^+V(y)<0$ on $D^1 - \{\bar{y}\}$.  
\item $D^2=\{y \in S_1|y_1<y_1(0), y_2 < d_2, y_3(0) < y_3 \leq d_1, \tilde{y}_2<\tilde{y}_2(0)\}$. The Lyapunov derivative is $D^+V(y)=-3ky_1<0$ in the interior $\mathbb{R}_{>0}$. In the non-equilibrium boundary set $D^2_a=\{y \in D^2|y_1=0, y \neq \bar{y}\}$, other variables satisfy $y_2+y_3-\tilde{y}_2 = d_1+d_2$, and combining $y_2 < d_2$, we have $y_3 > d_1+\tilde{y}_2 > d_1$ for all $t$, which contradicts with the fact that $y \in D^2$. Therefore, $y(t)$ will not run into $D^2_a$ and $D^+V(y) < 0 $ on $D^2$. 
\end{enumerate}

In the case of $d_2<0$, the equilibrium point is $\bar{y}=(0,0,d_1,\tilde{d}_2)$, where $\tilde{d_2}=-d_2$. So $\tilde{d}_2 < \tilde{y}_2(t) < \tilde{y}_2(0)$. The candidate Lyapunov function becomes $V(y)=y_1+y_2+d_1-y_3+\tilde{y}_2-\tilde{d}_2$, and the derivative is written as $\frac{\partial V(y)}{\partial t} = -ky_1-2ky_2\tilde{y}_2$ that is negative in the interior of the positive quadrant. For the boundary set $D^3_a = \{y \in \mathcal{S}_1|y_1=0, y_2=0\}$, one can obtain $y_3 - d_1 =\tilde{y}_2 -\tilde{d}_2$, which holds only when $y_3 = d_1, \tilde{y}_2 = \tilde{d}_2$ in $D^3_a$, i.e., equilibrium concentration. If $y(t)$ enters into another boundary set $D^3_b =\{y \in \mathcal{S}|y_1=0, \tilde{y}_2=0 \}$, it  may cause conflict with $\tilde{y}_2 > \tilde{d}_2 >0$. Therefore, $\frac{\partial V(y)}{\partial t} < 0$ on $\mathcal{S}_1 - \{\bar{y}\}$.
\end{proof}

\noindent \textbf{Proof of Proposition \ref{prop_finit_ex}:}

\begin{proof}
The stoichiometric compatibility class can be partitioned as $\mathcal{S}_1=D^1 \bigcup D^2$ in the case of $d_2>0$. If $\bm{y}(t)$ runs in $D^1$, the GAS guarantees that for any fixed time $T$, one can find $\epsilon>0$ such that the trajectory falls in the region $D^1_{\epsilon}=\{y(t) \in D^1| \Vert y(t) -\bar{y} \Vert_1 \geq \epsilon \}$ when $t \leq T$. In such case, there must exist $\epsilon_1, \epsilon_4$ such that $\vert y_1 \vert \geq \epsilon_1, \vert \tilde{y}_2\vert \geq \epsilon_4$. Thus, we can obtain $D^+V(y) < -K_1V(y)$, where $K_1=\min\{k,2k\epsilon_4,2kd_2,\frac{k\epsilon_1+2kd_2\epsilon_4}{2y_1(0)+\tilde{y}_2(0)+y_2(0)}\}$. If $y(t)$ runs in $D^2$, we can also find $D^+V(y)<-K_2V(\bm{y})$ for all $y(t) \in D^2_{\epsilon}=\{y(t) \in D^2| \Vert y(t) - \bar{y}_1 \Vert \geq \epsilon\}$ and $K_2 = \min\{3k,\frac{k\epsilon_1}{y_1(0)}\}$. A similar argument shows that $\frac{\partial V(y)}{\partial t} < -K_3 V(y)$ for all $y \in D^3_{\epsilon} = \{y \in \mathcal{S}_1 | \Vert y(t)-\bar{y} \Vert \geq \epsilon\}$ in the case of $d_2<0$, where $K_3 = \min\{k, 2k\tilde{d}_2,2k\epsilon_2,\frac{k\epsilon_1+2k\epsilon_2\tilde{d}_2}{2y_1(0)+\tilde{y}_2(0)+\tilde{d}_2} \}$.
These above imply that $\mathscr{M}^2$ can converge to any $\epsilon$-neighbourhood of the equilibrium at an exponential rate with the convergence order $K=\min\{K_1, K_2, K_3\}$ at least. Combining the \textit{comparison lemma}, it holds that in any out-epsilon region, there exits $M_y=V(y(0))$, s.t. $\Vert y(t) - \bar{y} \Vert < M_y e^{-K t}$, so finite-time exponential stability is present.
\end{proof}

\noindent \textbf{Proof of Proposition \ref{prop_bCRN2}:}

\begin{proof} The proof method is similar to Proposition \ref{prop_bCRN1}. We also use the polyhedral function $V_x(x)=\Vert x(t) - \bar{x}\Vert$ as the candidate Lyapunov function. Four cases for different compatibility classes must be discussed to analyze the GAS of equilibrium points. Every compatibility class should be partitioned into several sectors depending on the variable $x_2(t)$ and $x_3(t)$, and the GAS can be obtained by examining the behavior of $V(x)$ for nine regions in total. In addition, we can also deduce a positive constant $K_x$ such that $\dot{V}_x(x) < -K_x V_x(x)$, and make $M_x = V_x(x(0))$ for $\Vert x(t) - \bar{x} \Vert < M_x e^{-\tilde{K}t}$.
\end{proof}

\noindent \textbf{Proof of Lemma \ref{lemma_O27}:}

\begin{proof}
We have the dynamical equations of $\mathscr{C}^*$ as follows
\begin{gather*}
\begin{split}
    \dot{s}(t) & = g(s(t)),\\
    \dot{h}(t) & = s_1(t)s_2(t) - h(t). 
\end{split}
\end{gather*}
Since 
\begin{gather*}
\vert s_1(t) s_2(t) - \bar{s}_1\bar{s}_2 \vert = \vert (s_1(t)-\bar{s}_1)\bar{s}_2 + \bar{s}_1(s_2(t)-\bar{s}_2) \\
 +(s_1(t)-\bar{s}_1)(s_2(t)-\bar{s}_2) \vert,
\end{gather*}
we can obtain
\begin{gather*}
    \begin{split}
\vert h(t) - \bar{s}_1\bar{s}_2\vert & = \vert e^{-t} \int_{0}^{t} s_1(\tau) s_2(\tau)e^{\tau} \mathrm{d}\tau +e^{-t}h(0) -\bar{s}_1\bar{s}_2 \vert \\
& \leq \vert e^{-t} \int_{0}^{t} [s_1(\tau)s_2(\tau) - \bar{s}_1\bar{s}_2] e^{\tau} \mathrm{d}\tau \vert
 + G_s e^{-t}\\
 & \leq e^{-t} \int_{0}^{t} [\bar{s}_2 q_1 e^{(1-\beta_1) \tau} + \bar{s}_1 q_2 e^{(1-\beta_2) \tau} + q_1 q_2 \\
  & \cdot e^{(1-\beta_1 -  \beta_2) \tau}] \mathrm{d}\tau + G_s e^{-t},
    \end{split}
\end{gather*}
where $G_s = \vert h(0)-\bar{s}_1\bar{s}_2 \vert$. Let $a_1 = \frac{\bar{s}_2 q_1}{\vert 1-\beta_1 \vert}, a_2 = \frac{\bar{s}_1 q_2}{\vert 1-\beta_2 \vert}, a_3 = \frac{q_1 q_2}{\vert 1-\beta_1-\beta_2 \vert}, \tilde{a}_1=\mathbb{I}_{\{\beta_1\neq1\}}a_1 + \mathbb{I}_{\{\beta_1=1\}}\bar{s}_2q_1, \tilde{a}_2=\mathbb{I}_{\{\beta_2\neq1\}}a_2 + \mathbb{I}_{\{\beta_2=1\}}\bar{s}_1q_2$, where $\mathbb{I}$ is an indicator function. Based on the size of $\beta_1$ and $\beta_2$, we directly give the following characterization of exponential convergence $\vert h(t) - \bar{s}_1 \bar{s}_2 \vert < M_s e^{-\beta t}$ in different cases, which can be obtained by integration and scaling
\begin{enumerate}
    \item $\beta_1 + \beta_2 <1$: 
    $M_s = a_1+a_2+a_3+G_s, \beta = \min\{\beta_1, \beta_2\}$
    \item $\beta_1 + \beta_2 =1$: 
    $M_s = a_1+a_2+q_1q_2+G_s, \beta = \min\{\beta_1, \beta_2, 1-\frac{1}{e}\}$ 
    \item  $\beta_1 + \beta_2 > 1 (\beta_i = 1$ for some $i$): $M_s=\tilde{a}_1+\tilde{a}_2+a_3+G_s, \beta = \min\{1-\frac{1}{e}, \beta_1, \beta_2\}$
    \item  $\beta_1 + \beta_2 > 1 (\beta_i \neq 1$ for all $i$):
    $M_s=a_1+a_2+a_3+G_s, \beta = \min\{1, \beta_1, \beta_2\}$.
\end{enumerate}
In general, the exponential convergence of the parent network $\mathscr{C}^*$ can be guaranteed due to that of the independent subnetwork $\mathscr{C}, \tilde{\mathscr{C}}$.
\end{proof}

\noindent \textbf{Proof of Theorem \ref{thm_final}:}

\begin{proof}
The GES was assured by Lemma \ref{lemma_o13_F} and Lemma \ref{lemma_O27}. Besides, Lemma \ref{lemma_O27} implies that the equilibrium points of blocks from the root to leaf compute the value shaded by the same color in Fig.\ref{fig:SFP}, respectively, with catalysts $X_i$ storing results from previous modules. So ngBCRN finally obtains every negative gradient in two species $Par^+_i, Par^-_i$ in every training round. With $Par^{\pm}_i$ as the input signals species, the fixed point of the uBCRN system can be represented by
\begin{equation*}
    \bar{\Delta w}^{\pm} = \eta \cdot par^{\pm} (0), ~\bar{w}^{\pm} = \gamma^{\pm}(0) + \bar{\Delta w}^{\pm}
\end{equation*}
implies the successful update in every training round. Here, $w^{\pm}(t)=((w^{\pm}_1(t))^{\top}, (w^{\pm}_2(t))^{\top})^{\top} \in \mathbb{R}^{3 \times 3}$ with $w^{\pm}_{2}(t) \in \mathbb{R}^{1\times3}$ denoting concentrations of $W^{\pm}_5, W^{\pm}_6, B^{\pm}_3$ from the left to right, respectively, and $\Delta w^{\pm}(t) \in \mathbb{R}^{3 \times 3}$ has the same correspondence between incremental species numbers and concentration variable subscripts as that in $w^{\pm}(t)$. For $par^{\pm}(t) \in \mathbb{R}^{3 \times 3}$, $par^{\pm}_{1\cdot}(t), par^{\pm}_{2\cdot}(t), par^{\pm}_{3\cdot}(t)$ refers to concentration of $\{Par^{\pm}_1, Par^{\pm}_3, Par^{\pm}_7\}$, $\{Par^{\pm}_2, Par^{\pm}_4, Par^{\pm}_8\}$, $\{Par^{\pm}_5, Par^{\pm}_6, Par^{\pm}_9\}$, separately.
\end{proof} 
%

\noindent \textbf{Initial weight species concentration for ``OR'' and ``XOR''}
\begin{equation*}
w^+(0)=
\begin{bmatrix}
4.1968 &	1.8964 &	2.8458\\
1.4442 &	1.2924 &	3.0160\\
7.4030 &	8.3528 &	2.5640
\end{bmatrix} ~
w^-(0)=
\begin{bmatrix}
2 &	6	&6\\
2&	2	&2\\
6&	10	&2
\end{bmatrix},
\end{equation*}
and
\begin{equation*}
w^+(0)=
\begin{bmatrix}
3&	3	&2.5\\
4&	3	&2\\
2&	3	&2.5
\end{bmatrix} ~
w^-(0)=
\begin{bmatrix}
4 &	4	&1\\
3	&2&	2.5\\
1	&2&	4
\end{bmatrix}.~
\end{equation*}
The final fixed points $\bar{w}^+$ (left) and $\bar{w}^-$ (right) are
\begin{gather*}
\begin{bmatrix}
4.6620 &	1.9115 &3.4033 \\
2.5240 &	2.2642 &	5.3571\\
7.4863 &	8.7303 &	3.3334
\end{bmatrix}
\begin{bmatrix}
2.3766 &	6.0094 &	6.4861\\
3.2772&	3.2779&	4.5249 \\
6.0113 &	10.2436 &	2.3547
\end{bmatrix},
\end{gather*}
and 
\begin{gather*}
\begin{bmatrix}
3.2948 &	3.2884 &	3.0536\\
4.3969	&3.3891	& 2.8581\\
2.4997 &	3.4997	&3.2926
\end{bmatrix}
\begin{bmatrix}
4.2736 &	4.2709 &	1.4894 \\
3.3467 &	2.3475& 3.2918 \\
1.3883 &	2.3819 &4.6351 
\end{bmatrix}.
\end{gather*}

\section*{References}

\bibliographystyle{ieeetr}
\bibliography{main}

\end{document}